\documentclass[10pt, a4paper,american]{article}
\usepackage[pdftex]{graphicx}
\usepackage{amsmath, amssymb, amsthm, amsfonts}
\usepackage{float, color}
\usepackage{ascmac, fancybox}
\usepackage{bm}

\usepackage{bbm}

\usepackage[margin=1in]{geometry}

\makeatletter

\theoremstyle{plain}
\numberwithin{equation}{section}
\theoremstyle{plain}
\newtheorem{thm}{Theorem}[section]
\theoremstyle{plain}

\theoremstyle{definition}
\newtheorem{defn}[thm]{Definition}
\theoremstyle{definition}
\newtheorem{example}[thm]{Example}
\theoremstyle{plain}
\newtheorem{lem}[thm]{Lemma}
\theoremstyle{plain}
\newtheorem{cor}[thm]{Corollary}
\theoremstyle{definition}
\newtheorem{rem}[thm]{Remark}

\makeatother

\begin{document}
	
\title{Noncommutative Lebesgue decomposition and contiguity with applications in quantum statistics
}
\author{Akio Fujiwara%
	\thanks{fujiwara@math.sci.osaka-u.ac.jp}\\
	{Department of Mathematics, Osaka University}\\ 
	{Toyonaka, Osaka 560-0043, Japan}\\ \\
	and \\ \\
	Koichi Yamagata%
	\thanks{koichi.yamagata@uec.ac.jp}\\
	{Graduate School of Informatics and Engineering} \\
		{The University of Electro-Communications} \\
	{Chofu, Tokyo 182-8585, Japan}
}%
\date{}

\maketitle

\def\N{{\mathbb N}}
\def\Z{{\mathbb Z}}
\def\R{{\mathbb R}}  
\def\C{{\mathbb C}}

\def\X{{\mathcal{X}}}
\def\Y{{\mathcal{Y}}}
\def\Z{{\mathcal{Z}}}
\def\H{{\mathcal{H}}}
\def\L{{\mathcal{L}}}
\def\S{{\mathcal{S}}}
\def\M{{\mathcal{M}}}
\def\E{{\mathcal{E}}}
\def\B{{\mathcal{B}}}
\def\P{{\mathcal{P}}}

\def\a{{\alpha}}
\def\b{{\beta}}
\def\ga{{\gamma}}
\def\d{{\delta}}
\def\e{{\varepsilon}}
\def\l{{\lambda}}
\def\m{{\mu}}
\def\n{{\nu}}
\def\r{{\rho}}
\def\s{{\sigma}}
\def\t{{\tau}}
\def\th{{\theta}}
\def\x{{\xi}}
\def\y{{\eta}}
\def\z{{\zeta}}
\def\Th{{\Theta}}
\def\o{{\omega}}
\def\O{{\Omega}}

\def\span{{\rm span}}
\def\Tr{{\rm Tr\,}}
\def\tr{{\rm tr\,}}
\def\D{\mathcal{D}}
\def\G{\mathcal{G}}
\def\S{\mathcal{S}}
\def\H{\mathcal{H}}
\def\F{\mathcal{F}}
\def\K{\mathcal{K}}
\def\T{\mathcal{T}}
\def\P{\mathcal{P}}
\def\V{\mathcal{V}}
\def\W{\mathcal{W}}
\def\R{\mathbb{R}}
\def\Z{\mathbb{Z}}
\def\C{\mathbb{C}}

\def\braket#1#2{\left\langle #1|#2\right\rangle }
\def\dd#1{\frac{\mathrm{d}}{\mathrm{d}#1}}
\def\i{\sqrt{-1}}
\def\L{\mathcal{L}}
\def\N{\mathbb{N}}
\def\convd#1{\overset{#1}{\rightsquigarrow}}
\def\convp#1{\overset{#1}{\rightarrow}}
\def\convbe{\stackrel{\rightsquigarrow}{\rightsquigarrow}}
\def\A{\mathcal{A}}
\def\B{\mathcal{B}}
\def\Re{{\rm Re\,}}
\def\Im{{\rm Im\,}}
\def\convq#1{\underset{#1}{\rightsquigarrow}}
\def\ratio#1#2{\mathcal{R}\left(#1\middle|#2\right)}
\def\Ratio{\mathcal{R}}
\def\trans{\,^{t}}
\def\holevo#1#2#3{C_{#1}\left(#2,#3\right)}
\def\CCR#1{{\rm CCR}\left(#1\right)}
\def\supp{{\rm supp\,}}
\def\rank{{\rm rank\,}}
\def\sand#1#2{\left<#1,#2,#1 \right>}

\newcommand{\ket}[1]{\left | #1 \right \rangle}
\newcommand{\bra}[1]{\left \langle #1 \right |}
\newcommand{\bracket}[2]{\left \langle #1 \left | #2 \right\rangle\right.}
\newcommand{\Map}[2]{#1 \rightarrow  #2}
\newcommand{\cent}[1]{\begin{center} #1 \end{center}}
\newcommand{\rest}{{\t=\r,\s=\r}}
\newcommand{\leftDummy}{\left. \phantom{\prod\!\!\!\!\!\!\!\!\!}}

\newcommand{\argmax}{\mathop{\rm arg~max}\limits}
\newcommand{\argmin}{\mathop{\rm arg~min}\limits}


\begin{abstract}
	We herein develop a theory of contiguity in the quantum domain based upon a novel quantum analogue of the Lebesgue decomposition.
	The theory thus formulated is pertinent to the weak quantum local asymptotic normality introduced in the previous paper [Yamagata, Fujiwara, and Gill, \textit{Ann. Statist.},  \textbf{41}  (2013) 2197-2217.],  
	yielding substantial enlargement of the scope of quantum statistics. 
\end{abstract}

\section{Introduction}

Quantum statistics is a rapidly growing field of research in quantum information science.
When we consider the future direction of the field, 
we may learn much from the history of classical statistics.
One of the deepest achievements in mathematical statistics is the theory of local asymptotic normality introduced by Le Cam \cite{LeCam:1966}. 
A sequence $\left\{P_{\theta}^{(n)} \left|\;\theta\in\Theta\subset\R^{d}\right.\right\}$ 
of $d$-dimensional parametric models, 
each comprising probability measures on a measurable space $(\Omega^{(n)}, \F^{(n)})$, 
is said to be {\em locally asymptotically normal} (LAN) at $\theta_{0}\in\Theta$ 
(in the ``weak'' sense) 
if there exist a sequence $\Delta^{(n)}=(\Delta_1^{(n)},\,\dots,\,\Delta_d^{(n)})$ of $d$-dimensional random vectors and a $d\times d$ real symmetric positive definite matrix $J$ such that $\Delta^{(n)}\convd {0} N(0,J)$ and
\begin{equation}\label{eqn:c-LAN}
\log\frac{dP_{\theta_{0}+h/\sqrt{n}}^{(n)}}{dP_{\theta_{0}}^{(n)}}
=h^{i}\Delta_{i}^{(n)}-\frac{1}{2}h^{i}h^{j}J_{ij}+o_{P_{\theta_0}^{(n)}}(1),
\qquad 
(h\in\R^{d}).
\end{equation}
Here the arrow $\convd {h}$ stands for the convergence in distribution under $P_{\theta_{0}+h/\sqrt{n}}^{(n)}$, the remainder term $o_{P_{\theta_0}^{(n)}}(1)$ converges in probability to zero under $P_{\theta_0}^{(n)}$, and Einstein's summation convention is used.

The notion of local asymptotic normality provides a useful tool to cope with various statistical models in a unified manner by reducing them to relevant Gaussian shift models in the asymptotic limit.
Observe that the expansion \eqref{eqn:c-LAN} is similar in form to the log-likelihood ratio of the Gaussian shift model:
\[
\log \frac{dN(h,J^{-1})}{dN(0,J^{-1})} (X^1,\dots,X^d)
=h^i (X^j J_{ij})-\frac{1}{2}h^ih^j J_{ij}.
\]
This similarity suggests a deep relationship between the models $\{ P_{\theta_{0}+h/\sqrt{n}}^{(n)} \mid h\in\R^d\}$ and $\{N(h, J^{-1})\mid h\in\R^d\}$. 
In order to put the similarity to practical use, Le Cam introduced the notion of contiguity \cite{LeCam:1966}. 
A sequence $Q^{(n)}$ of probability measures is called {\em contiguous} with respect to another sequence $P^{(n)}$ of probability measures, denoted $Q^{(n)} \vartriangleleft P^{(n)}$, if $P^{(n)}(A^{(n)})\to 0$ implies $Q^{(n)}(A^{(n)})\to 0$ for any sequence $A^{(n)}$ of measurable sets. 
An important conclusion pertinent to the notion of contiguity is the following theorem, which is usually referred to as Le Cam's third Lemma: if $Q^{(n)} \vartriangleleft P^{(n)}$ and 
\[
 \left(X^{(n)},\frac{dQ^{(n)}}{dP^{(n)}} \right) \convd{P^{(n)}} (X,V), 
\]
then $X^{(n)} \convd{Q^{(n)}} L$, where $L$ is the law defined by $L(B):=E[1_B (X) V]$. 
Since the local asymptotic normality \eqref{eqn:c-LAN} entails mutual contiguity 
$P_{\theta_{0}+h/\sqrt{n}}^{(n)} \vartriangleleft \vartriangleright P_{\theta_{0}}^{(n)}$,
Le Cam's third lemma proves that $X^{(n)j}:=(J^{-1})^{jk} \Delta^{(n)}_k$ exhibits 
$X^{(n)}\convd {h} N(h,J^{-1})$. 
This gives a precise meaning of the statement that the model $\{ P_{\theta_{0}+h/\sqrt{n}}^{(n)} \,|\, h\in\R^d\}$ satisfying \eqref{eqn:c-LAN} is statistically similar to the Gaussian shift model $\{N(h, J^{-1})\,|\, h\in\R^d\}$. 

Note that such an interpretation is realized in the asymptotic framework. 
A measure theoretic counterpart of Le Cam's third lemma is the identity 
$dQ=({dQ}/{dP}) dP$, which is valid when $Q$ is absolutely continuous to $P$.
In the non-asymptotic framework, the likelihood ratio ${dP_{\theta_{0}+h/\sqrt{n}}^{(n)}}/{dP_{\theta_{0}}^{(n)}}$ carries full information about the measure $P_{\theta_{0}+h/\sqrt{n}}^{(n)}$ only when $P_{\theta_{0}+h/\sqrt{n}}^{(n)}$ is absolutely continuous to $P_{\theta_{0}}^{(n)}$. 
This fact demonstrates the differences between the contiguity and the absolute continuity,  highlighting the notable flexibility and usefulness of the notion of contiguity when it is used in conjunction with the weak LAN.

Extending the notion of local asymptotic normality to the quantum domain was pioneered by Gu\c{t}\u{a} and Kahn \cite{GutaQLANfor2, GutaQLANforD}.
They proved that, given a quantum parametric model $\S(\C^D)=\{ \rho_\theta>0 \mid \theta\in\Theta\subset\R^{D^2-1} \}$ comprising the totality of faithful density operators on a $D$-dimensional Hilbert space and a point $\theta_0$ on the parameter space $\Theta$ such that $\r_{\theta_0}$ is nondegenerate  (i.e., every eigenvalue of $\rho_{\theta_0}$ is simple), there exist, for any compact subset $K\;(\subset \R^{D^2-1})$, 
quantum channels $S_n$ and  $T_n$ such that
\[
 \lim_{n\to\infty}\sup_{h\in K} \left\| \sigma_h-T_n(\rho_{\theta_0+h/\sqrt{n}}^{\otimes n}) \right\|_1 =0,
 \quad\mbox{and}\quad
 \lim_{n\to\infty}\sup_{h\in K} \left\| S_n(\sigma_h)-\rho_{\theta_0+h/\sqrt{n}}^{\otimes n} \right\|_1 =0,
\]
where $\{\sigma_h \mid h\in\R^{D^2-1}\}$ is a family of density operators of a quantum Gaussian shift model $N(h, J^{-1})$ with $J$ being the RLD Fisher information matrix of $\r_\theta$ at $\theta_0\in\Theta$.
(See Appendix \ref{app:q-Gaussian} for a brief account of quantum Gaussian states.)

Note that this formulation is not a direct analogue of the weak LAN defined by \eqref{eqn:c-LAN}; in particular, the convergence to a quantum Gaussian shift model is evaluated not by the convergence in distribution but by the convergence in trace norm.
In this sense, their formulation could be called a ``strong'' q-LAN,  
(cf.  \cite[Chapter 10]{LeCam:Book}).

Gu\c{t}\u{a} and Kahn's theorem in terms of the strong q-LAN was so powerful that it was applied to the study of asymptotic quantum parameter estimation problems in \cite{YangCH}.
However, the strong q-LAN after Gu\c{t}\u{a} and Kahn is not fully satisfactory because it is applicable only to i.i.d. extensions of a quantum statistical model around a nondegenerate reference state $\rho_{\theta_0}$. 
It is natural to seek a more flexible formulation that is applicable to non-i.i.d. cases with possibly degenerate reference states. 
In \cite{GutaQLANweak}, they tried a different approach to a ``weak'' q-LAN via the Connes cocycle derivative, which was sometimes regarded as a proper quantum analogue of the likelihood ratio.
However, they did not establish an asymptotic expansion formula which would be directly analogous to \eqref{eqn:c-LAN} in the classical LAN.

A different approach to a weak q-LAN was put forward in \cite{YFG}, in which a sequence of quantum statistical models comprising mutually absolutely continuous density operators was treated.
Here, density operators $\r$ and $\s$ on a finite dimensional Hilbert space are said to be 
{\em mutually absolutely continuous}, $\r\sim\s$ in symbols, if there exists a Hermitian operator $\L$ that satisfies
\[
  \s=e^{\frac{1}{2}\L}\r e^{\frac{1}{2}\L}.
\]
The operator $\L$ satisfying this relation is called (a version of) the {\em quantum log-likelihood ratio}. 
When the reference states $\r$ and $\s$ need to be specified, $\L$ is denoted as $\L(\s |\r)$, so that 
\[
  \s=e^{\frac{1}{2}\L(\s |\r)} \r e^{\frac{1}{2}\L(\s |\r)}.
\]
For example, when both $\r$ and $\s$ are strictly positive, the quantum log-likelihood ratio is uniquely given by
\[
 \L(\s |\r)=2\log\left(\s \# \r^{-1} \right).
\]
Here, $\#$ denotes the operator geometric mean \cite{{Bhatia},{KuboAndo}}: 
for strictly positive operators $A$ and $B$, the operator geometric mean $A\# B$ 
is defined as the unique positive operator $X$ that satisfies the equation $B=XA^{-1}X$, 
and is explicitly given by
\[ A\# B=\sqrt{A}\sqrt{\sqrt{A^{-1}} B \sqrt{A^{-1}}\,} \sqrt{A}. \]

The theory of weak q-LAN developed in \cite{YFG} was successfully applied to quantum statistical models satisfying only some mild regularity conditions, and clarified that the Holevo bound was asymptotically achievable. 
However, this formulation, too, is not fully satisfactory because it is applicable only to quantum statistical models that comprises mutually absolutely continuous density operators.
This is in good contrast to the classical definition \eqref{eqn:c-LAN}, in which mutual absolute continuity for the model was not assumed \cite{Vaart}. 
The key idea behind this classical formulation is the use of the Radon-Nikodym density, 
or more fundamentally, the use of the Lebesgue decomposition of $P_{\theta_{0}+h/\sqrt{n}}^{(n)}$ with respect to $P_{\theta_{0}}^{(n)}$.
Thus, in order to extend such a flexible formulation to the quantum domain, we must invoke an appropriate quantum counterpart of the Lebesgue decomposition.
Several noncommutative analogues of 
the Lebesgue decomposition and/or the Radon-Nikodym derivative 
have been devised, e.g., \cite{{Blackadar}, {Connes}, {ConnesBook}, {Dye}, {KadisonRingrose}, {Kosaki}, {OhyaPetz}, {Parthasarathy}, {PedersenTakesaki}, {PetzBook}, {Sakai}, {Takesaki}, {UmegakiOH}}.
However, each of them has its own scope, and to the best of our knowledge, no appropriate quantum counterpart that is applicable to the theory of weak q-LAN has been established. 

The objective of the present paper is threefold:
Firstly, we devise a novel quantum analogue of the Lebesgue decomposition that is pertinent to the framework of weak q-LAN introduced in the previous paper \cite{YFG}. 
Secondly, we develop a theory of contiguity in the quantum domain based on the novel quantum Lebesgue decomposition.
One of the remarkable achievements of the theory is the abstract version of Le Cam's third lemma (Theorem \ref{thm:qLeCam3}). 
Finally, we apply the theory of quantum contiguity to weak q-LAN, 
yielding substantial enlargement of the scope of q-LAN as compared with the previous paper \cite{YFG}. 

The present paper is organized as follows. 
In Section \ref{sec:acANDsing}, we extend the notions of absolute continuity and singularity to the quantum domain in order that they are fully consistent with the notion of mutual absolute continuity introduced in \cite{YFG}. 
In Section \ref{sec:Lebesgue}, we formulate a quantum Lebesgue decomposition based on the quantum absolute continuity and singularity introduced in Section \ref{sec:acANDsing}.
In Section \ref{sec:contiguity}, we develop a theory of quantum contiguity by taking full advantage of the novel quantum Lebesgue decomposition established in Section \ref{sec:Lebesgue}. 
In Section \ref{sec:convergenceInLaw}, we introduce the notion of convergence in distribution in terms of the quasi-characteristic function, and prove a noncommutative version of the L\'evy-Cram\'er continuity theorem under the ``sandwiched'' convergence in distribution, which plays a key role in the subsequent discussion.
In Section \ref{sec:lecam3}, we prove a quantum counterpart of the Le Cam third lemma. This achievement manifests the validity of the novel quantum Lebesgue decomposition and quantum contiguity as well as the notion of sandwiched convergence in distribution.
In Section \ref{sec:example}, we give some illustrative examples that demonstrate the flexibility and applicability of the present formulation in asymptotic quantum statistics, 
including a quantum contiguity version of the Kakutani dichotomy, 
and enlargement of the scope of q-LAN.
Section \ref{sec:conclusion} is devoted to brief concluding remarks. 
For the reader's convenience, some additional material is presented in Appendix, including the quantum Gaussian states, and a noncommutative L\'evy-Cram\'er continuity theorem.

\section{Absolute continuity and singularity}\label{sec:acANDsing}

Given positive operators $\r$ and $\s$ on a (finite dimensional) Hilbert space $\H$ with $\r\neq 0$, 
let $\sigma\!\!\downharpoonleft_{\supp\rho}$ denote the {\em excision} of $\s$ relative to $\r$ by the operator on the subspace $\supp\rho:=(\ker\rho)^\perp$ of $\H$ defined by
\[ \sigma\!\!\downharpoonleft_{\supp\rho}:=\iota_\rho^*\, \sigma\, \iota_\rho, \]
where $\iota_\rho: \supp\rho\hookrightarrow \H$ is the inclusion map. 
More specifically, let
\begin{equation}\label{eqn:blockRS}
\rho=\begin{pmatrix} \rho_0 & 0\\ 0 & 0 \end{pmatrix},
\qquad
\sigma=\begin{pmatrix} \sigma_0 & \alpha\\ \alpha^* & \beta
\end{pmatrix}
\end{equation}
be a simultaneous block matrix representations of $\r$ and $\s$, where $\r_0>0$. 
Then the excision $\sigma\!\!\downharpoonleft_{\supp\rho}$ is nothing but the operator represented by the $(1,1)$th block $\s_0$ of $\s$. 
The notion of excision was exploited in \cite{YFG}. 
In particular, it was shown that $\r$ and $\s$ are mutually absolutely continuous if and only if 
\[
 \sigma\!\!\downharpoonleft_{\supp\rho}>0 \quad\mbox{and}\quad \rank\rho=\rank\sigma,
\]
or equivalently, if and only if
\begin{equation}\label{eqn:mutuallyAC}
 \sigma\!\!\downharpoonleft_{\supp\rho}>0 \quad\mbox{and}\quad \rho\!\!\downharpoonleft_{\supp\sigma}>0.
\end{equation}

Now we introduce noncommutative analogues of the notions of absolute continuity and singularity that played essential roles in the classical measure theory. 
Given positive operators $\r$ and $\s$, we say $\r$ is {\em singular} with respect to $\s$, denoted $\r\perp\s$, if
\[
 \sigma\!\!\downharpoonleft_{\supp\rho}=0.
\]
The following lemma implies that the relation $\perp$ is symmetric; 
this fact allows us to say that $\r$ and $\s$ are mutually singular, as in the classical case. 

\begin{lem}\label{lem:1}
For nonzero positive operators $\r$ and $\s$, the following are equivalent. 
\begin{itemize}
\item[{\rm (a)}]  $\r\perp\s$. 
\item[{\rm (b)}]  $\supp\rho \perp \supp\sigma$.
\item[{\rm (c)}]  $\Tr \r \s=0$.
\end{itemize}
\end{lem}

\begin{proof}
Let us represent $\r$ and $\s$ in the form \eqref{eqn:blockRS}. 
Then, (a) is equivalent to $\s_0=0$. 
In this case, the positivity of $\s$ entails that the off-diagonal blocks $\a$ and $\a^*$ of $\s$ also vanish, and $\s$ takes the form
\[ 
\sigma=\begin{pmatrix}
0 & 0\\
0 & \beta
\end{pmatrix}.
\]
This implies (b). 
Next, (b) $\Rightarrow$ (c) is obvious.  
Finally, assume (c). 
With the representation \eqref{eqn:blockRS}, this is equivalent to $\Tr \r_0 \s_0=0$.
Since $\r_0>0$, we have $\s_0=0$, proving (a).
\end{proof}

We next introduce the notion of absolute continuity. 
Given positive operators $\r$ and $\s$, we say $\r$ is {\em absolutely continuous} with respect to $\s$, denoted $\r\ll\s$,  if
\[
 \sigma\!\!\downharpoonleft_{\supp\rho}>0.
\]

Some remarks are in order. 
Firstly, the above definition of absolute continuity is consistent with the definition of mutual absolute continuity: 
in fact, as demonstrated in \eqref{eqn:mutuallyAC}, $\r$ and $\s$ are mutually absolutely continuous if and only if both $\r\ll\s$ and $\s\ll\r$ hold. 
Secondly, $\r\ll\s$ is a much weaker condition than $\supp\r\subset\supp\s$: 
this makes a striking contrast to the classical measure theory. 
For example, pure states $\r=\ket{\psi}\bra{\psi}$ and $\s=\ket{\xi}\bra{\xi}$ are mutually absolutely continuous if and only if $\braket{\xi}{\psi}\neq 0$, (see \cite[Example 2.3]{YFG}).

The following lemma plays a key role in the next section, leading to a novel noncommutative Lebesgue decomposition.

\begin{lem}\label{lem:2}
For nonzero positive operators $\r$ and $\s$, the following are equivalent. 
\begin{itemize}
\item[{\rm (a)}]  $\r\ll\s$. 
\item[{\rm (b)}]  $\exists R> 0$ such that $\s\ge R\r R$.
\item[{\rm (c)}]  $\exists R> 0$ such that $\r\le R\s R$.
\item[{\rm (d)}]  $\exists R\ge 0$ such that $\r=R\s R$.
\item[{\rm (e)}]  $\exists R\ge 0$ such that $\r \geq R \s R$ and $\Tr \r =\Tr \s R^2$.
\end{itemize}
\end{lem}

\begin{proof}
We first prove (a) $\Rightarrow$ (b). 
Let
\[ 
\rho=\begin{pmatrix} \rho_0 & 0\\ 0 & 0 \end{pmatrix},
\qquad
\sigma=\begin{pmatrix} \sigma_0 & \alpha\\ \alpha^* & \beta
\end{pmatrix}
\]
where $\r_0>0$. 
Since $\s_0=\s\!\!\downharpoonleft_{\supp\r}>0$, 
the matrix $\s$ is further decomposed as
\[
\sigma=
E^*
\begin{pmatrix}
\sigma_0 & 0\\
0 & \beta-\alpha^* \sigma_0^{-1} \alpha
\end{pmatrix}
E,
\qquad
E:=
\begin{pmatrix}
I & \sigma_0^{-1} \alpha\\
0& I
\end{pmatrix}.
\]
Note that, since $\s\ge 0$ and $E$ is full-rank, we have
\begin{equation}\label{eqn:sigma22}
\beta-\alpha^* \sigma_0^{-1} \alpha\ge 0.
\end{equation} 
Now we set
\[
R:=E^* \begin{pmatrix} X & 0\\ 0 & \ga \end{pmatrix} E,
\]
where $X:=\s_0 \# \r_0^{-1}$, 
and $\ga$ is an arbitrary strictly positive operator. 
Then 
\begin{eqnarray*}
R\r R
&=& 
E^* \begin{pmatrix} X & 0\\ 0 & \ga \end{pmatrix} E
\begin{pmatrix} \rho_0 & 0\\ 0 & 0 \end{pmatrix}
E^* \begin{pmatrix} X & 0\\ 0 & \ga \end{pmatrix} E \\
&=& 
E^* \begin{pmatrix} X & 0\\ 0 & \ga \end{pmatrix}
\begin{pmatrix} \rho_0 & 0\\ 0 & 0 \end{pmatrix}
\begin{pmatrix} X & 0\\ 0 & \ga \end{pmatrix} E \\
&=& 
E^* \begin{pmatrix} X \rho_0 X & 0\\ 0 & 0 \end{pmatrix} E \\
&=& 
E^* \begin{pmatrix} \s_0 & 0\\ 0 & 0 \end{pmatrix} E \\
&\le& 
E^* \begin{pmatrix} \s_0 & 0\\ 0 & \beta-\alpha^* \sigma_0^{-1} \alpha \end{pmatrix} E
=
\s.
\end{eqnarray*}
Here, the inequality is due to \eqref{eqn:sigma22}. 
Since $R>0$, we have (b).

We next prove (b) $\Rightarrow$ (a). 
Due to assumption, there is a positive operator $\t\ge 0$ such that 
\[ \s=R\r R+\t. \]
Let
\[ 
\rho=\begin{pmatrix} \rho_0 & 0\\ 0 & 0 \end{pmatrix},\qquad
R=\begin{pmatrix} R_0 & R_1\\ R_1^* & R_2 \end{pmatrix},\qquad
\t=\begin{pmatrix} \t_0 & \t_1\\ \t_1^* & \t_2 \end{pmatrix},
\]
where $\r_0>0$. 
Then
\[
 \s=\begin{pmatrix} R_0\r_0 R_0+\t_0 & R_0\r_0 R_1 +\t_1
 	\\ R_1^*\r_0 R_0+\t_1^* & R_1^*\r_0 R_1+\t_2 \end{pmatrix}
\]
and
\[
 \s\!\!\downharpoonleft_{\supp\r}=R_0\r_0 R_0+\t_0. 
\]
Since $R_0>0$ and $\t_0\ge 0$, we have $\s\!\!\downharpoonleft_{\supp\r}>0$.

For the proof of (a) $\Rightarrow$ (d), let
\[ 
\r=\begin{pmatrix} \r_0 & 0\\ 0 & 0 \end{pmatrix},
\qquad
\s=\begin{pmatrix} \s_0 & \a \\ \a^* & \b
\end{pmatrix}, 
\]
where $\r_0>0$. Since $\s_0=\s\!\!\downharpoonleft_{\supp\r}>0$, 
\[
R:=\begin{pmatrix} \r_0 \# \s_0^{-1} & 0\\ 0 & 0 \end{pmatrix}
\]
is a well-defined positive operator satisfying
\[
 \r=R \s R. 
\]
This proves (d). 

For (d) $\Rightarrow$ (a), 
let the positive operator $R$ in $\r=R \s R$ be represented as
\[ 
R=\begin{pmatrix} R_0 & 0\\ 0 & 0 \end{pmatrix},
\]
where $R_0>0$, and accordingly, let us represent $\r$ and $\s$ as
\[ 
\r=\begin{pmatrix} \r_0 & \r_1\\ \r_1^* & \r_2 \end{pmatrix},
\qquad
\s=\begin{pmatrix} \s_0 & \s_1 \\ \s_1^* & \s_2
\end{pmatrix}.
\]
The relation $\r=R \s R$ is then reduced to 
\[ 
\begin{pmatrix} \r_0 & \r_1\\ \r_1^* & \r_2 \end{pmatrix}
=
\begin{pmatrix} R_0 \s_0 R_0 & 0 \\ 0 & 0
\end{pmatrix}.
\]
This implies that $\supp\r=\supp\r_0$ and $\r_0\sim\s_0$. 
Consequently, 
\[
 \s\!\!\downharpoonleft_{\supp\r}
 =\s\!\!\downharpoonleft_{\supp\r_0}
 =\s_0\!\!\downharpoonleft_{\supp\r_0}
 \,>0.
\]
In the last inequality, we used the fact that $\r_0\sim\s_0$ implies $\r_0\ll\s_0$. 

Now that (b) $\Leftrightarrow$ (c) and (d) $\Leftrightarrow$ (e) are obvious, the proof is complete.
\end{proof}

\section{Lebesgue decomposition}\label{sec:Lebesgue}

In this section, we extend the Lebesgue decomposition to the quantum domain.

\subsection{Case 1:  when $\s \gg \r$}

To elucidate our motivation, 
let us first treat the case when $\s\gg \r$. 
In Lemma \ref{lem:2}, we found the following characterization: 
\[
 \s\gg\r \;\Longleftrightarrow\; \exists R> 0
 \mbox{ such that $\s\ge R\r R$}.
\]
Note that such an operator $R$ is not unique. 
For example, suppose that $\s\ge R_1\r R_1$ holds for some $R_1>0$.
Then for any $t\in (0,1]$, the operator $R_t:=t R_1$ is strictly positive and satisfies $\s\ge R_t \r R_t$. 
It is then natural to seek, if any, the ``maximal'' operator of the form $R\r R$ that is packed into $\s$.  
Put differently, letting $\t:=\s-R\r R$, we want to find the ``minimal'' positive operator $\t$ that satisfies 
\begin{equation}\label{eqn:decomp0}
 \s=R\r R+\t,
\end{equation}
where $R>0$.
This question naturally leads us to a noncommutative analogue of the Lebesgue decomposition, 
in that 
a positive operator $\t$ satisfying \eqref{eqn:decomp0} is regarded as minimal if $\t \perp \r$. 

In the proof of Lemma \ref{lem:2}, we found the following decomposition: 
\begin{eqnarray*}
\s
&=&E^* \begin{pmatrix} \s_0 & 0\\ 0 & \beta-\alpha^* \sigma_0^{-1} \alpha \end{pmatrix} E \\
&=&E^* \begin{pmatrix} \s_0 & 0\\ 0 & 0  \end{pmatrix} E
+E^* \begin{pmatrix} 0 & 0\\ 0 & \beta-\alpha^* \sigma_0^{-1} \alpha \end{pmatrix} E \\
&=&R\r R+\begin{pmatrix} 0 & 0\\ 0 & \beta-\alpha^* \sigma_0^{-1} \alpha \end{pmatrix}
\end{eqnarray*}
where
\[
\rho=\begin{pmatrix} \r_0 & 0\\ 0 & 0 \end{pmatrix},
\quad
\sigma=\begin{pmatrix} \s_0 & \a\\ \a^* & \beta \end{pmatrix},
\quad
E:=\begin{pmatrix} I & \s_0^{-1} \a\\ 0& I \end{pmatrix},
\quad 
R=E^* \begin{pmatrix} \s_0 \# \r_0^{-1} & 0\\ 0 & I \end{pmatrix} E
\]
with $\r_0>0$ and $\s_0>0$. 
Since
\[
 \begin{pmatrix} \r_0 & 0\\ 0 & 0 \end{pmatrix}
 \perp
 \begin{pmatrix} 0 & 0\\ 0 & \beta-\alpha^* \sigma_0^{-1} \alpha \end{pmatrix},
\]
we have the following decomposition:
\begin{equation}\label{eqn:LebDec1}
 \s=\s^{ac}+\s^\perp,
\end{equation}
where
\begin{equation}\label{eqn:LebDec1AC}
\s^{ac}:=R\r R=\begin{pmatrix} \s_0 & \a \\ \a^* & \alpha^* \sigma_0^{-1} \alpha \end{pmatrix}
\end{equation}
is the (mutually) {\em absolutely continuous part }of $\s$ with respect to  $\r$, 
and 
\begin{equation}\label{eqn:LebDec1singular}
\s^\perp:=\begin{pmatrix} 0 & 0\\ 0 & \beta-\alpha^* \sigma_0^{-1} \alpha \end{pmatrix}
\end{equation}
is the {\em singular part} of $\s$ with respect to  $\r$. 

We may call the decomposition \eqref{eqn:LebDec1} a {\em quantum Lebesgue decomposition} for the following reasons. 
Firstly, 
although \eqref{eqn:LebDec1} was defined by using a simultaneous block matrix representation of $\r$ and $\s$, which has an arbitrariness of unitary transformations of the form $U_1\oplus U_2$, 
the matrices \eqref{eqn:LebDec1AC} and \eqref{eqn:LebDec1singular} are covariant under those unitary transformations, and hence 
the operators $\s^{ac}$ and $\s^\perp$ are well-defined regardless of the arbitrariness of the block matrix representation. 
Secondly, 
the decomposition \eqref{eqn:LebDec1} is unique, as the following lemma asserts. 

\begin{lem}\label{lem:uniqueness1}
Suppose $\s\gg\r$. Then the decomposition
\begin{equation}\label{eqn:qLebesgue1}
 \s=\s^{ac}+\s^\perp\qquad (\s^{ac} \ll \r,\;\s^\perp \perp \r)
\end{equation}
is uniquely given by \eqref{eqn:LebDec1AC} and \eqref{eqn:LebDec1singular}.
\end{lem}

\begin{proof}
We show that the decomposition
\begin{equation}\label{eqn:uniqueness1}
 \s=R\r R+\t\qquad (R \ge 0,\,\t\ge 0,\; \t \perp \r)
\end{equation}
is unique. 
Let
\[
\rho=\begin{pmatrix} \r_0 & 0\\ 0 & 0 \end{pmatrix},
\quad
\sigma=\begin{pmatrix} \s_0 & \a\\ \a^* & \beta \end{pmatrix}
\]
with $\r_0>0$. 
Due to assumption $\r\ll\s$, we have $\s_0>0$. 
Let
\[
E:=\begin{pmatrix} I & \s_0^{-1} \a\\ 0& I \end{pmatrix}.
\]
Since $E$ is invertible, the operator $R$ appeared in \eqref{eqn:uniqueness1} is represented in the form
\[
R=E^* \begin{pmatrix} R_0 & R_1 \\ R_1^* & R_2 \end{pmatrix} E.
\]
With this representation
\begin{eqnarray*}
R\r R
&=& 
E^* \begin{pmatrix} R_0 & R_1 \\ R_1^* & R_2  \end{pmatrix} E
\begin{pmatrix} \rho_0 & 0\\ 0 & 0 \end{pmatrix}
E^* \begin{pmatrix} R_0 & R_1 \\ R_1^* & R_2 \end{pmatrix} E \\
&=& 
E^* 
\begin{pmatrix} R_0 \r_0 R_0 & R_0\r_0 R_1 \\ R_1^*\r_0 R_0 & R_1^*\r_0 R_1  \end{pmatrix} 
E\\
&\le& 
\s
=E^* \begin{pmatrix} \s_0 & 0\\ 0 & \beta-\alpha^* \sigma_0^{-1} \alpha \end{pmatrix} E.
\end{eqnarray*}
Here, the inequality is due to \eqref{eqn:uniqueness1}. 
Let us denote the singular part $\t$ as
\[
 \t=\begin{pmatrix} 0 & 0 \\ 0 & \t_0 \end{pmatrix}
 =E^* \begin{pmatrix} 0 & 0 \\ 0 & \t_0 \end{pmatrix} E.
\]
Then the decomposition \eqref{eqn:uniqueness1} is equivalent to 
\begin{equation}\label{eqn:uniqueness1-1}
 \begin{pmatrix} \s_0 & 0\\ 0 & \beta-\alpha^* \sigma_0^{-1} \alpha \end{pmatrix} 
 = \begin{pmatrix} R_0 \r_0 R_0 & R_0\r_0 R_1 \\ R_1^*\r_0 R_0 & R_1^*\r_0 R_1  \end{pmatrix} 
 + \begin{pmatrix} 0 & 0 \\ 0 & \t_0 \end{pmatrix}.
\end{equation}
Comparison of the $(1,1)$th blocks of both sides yields 
$R_0=\s_0 \# \r_0^{-1}$. 
Since this $R_0$ is strictly positive, comparison of other blocks of \eqref{eqn:uniqueness1-1} further yields
\[
 R_1=0\quad\mbox{and}\quad \t_0=\beta-\alpha^* \sigma_0^{-1} \alpha.
\]
Consequently, the singular part $\t$ is uniquely determined by \eqref{eqn:LebDec1singular}. 
\end{proof}

An immediate consequence of Lemma \ref{lem:uniqueness1} is the following

\begin{cor}\label{cor:uniqueness1}
When $\s\gg\r$, the absolutely continuous part $\s^{ac}$ of the quantum Lebesgue decomposition \eqref{eqn:qLebesgue1}
is in fact mutually absolutely continuous to $\r$, i.e., $\s^{ac}\sim\r$.
\end{cor}

Note that the operator $R_2$ appeared in the proof of Lemma \ref{lem:uniqueness1} is arbitrary as long as it is positive. 
Because of this arbitrariness, we can take the operator $R$ in \eqref{eqn:uniqueness1} to be strictly positive. 
This gives an alternative view of Corollary \ref{cor:uniqueness1}.

\subsection{Case 2:  generic case}

Let us extend the quantum Lebesgue decomposition \eqref{eqn:qLebesgue1} to a generic case when 
$\r$ is not necessarily absolutely continuous with respect to $\s$. 
When $\r$ and $\s$ are mutually singular, we just let $\s^{ac}=0$ and $\s^\perp=\s$. 
We therefore assume in the rest of this section that $\r$ and $\s$ are not mutually singular.

Given positive operators $\r$ and $\s$ that satisfy $\r\not\perp\s$, 
let $\H=\H_1\oplus\H_2\oplus\H_3$ be the orthogonal direct sum decomposition defined by
\[ 
 \H_1:=\ker\left(\sigma\!\!\downharpoonleft_{\supp\rho}\right),\qquad 
 \H_2:=\supp\left(\sigma\!\!\downharpoonleft_{\supp\rho}\right),\qquad 
 \H_3:=\ker\r. 
\]
Then $\r$ and $\s$ are represented in the form of block matrices as follows:
\begin{equation}\label{eqn:blockMatrix}
\rho=\begin{pmatrix} \r_2 & \r_1 & 0\\ \r_1^* & \r_0 & 0 \\ 0 & 0 & 0 \end{pmatrix}, 
\qquad
\sigma=\begin{pmatrix} 0 & 0 & 0 \\ 0 & \s_0 & \a \\ 0 & \a^* & \beta \end{pmatrix},
\end{equation}
where 
\[
\begin{pmatrix} \r_2 & \r_1 \\ \r_1^* & \r_0 \end{pmatrix}>0,\qquad \s_0>0.
\]
Note that when 
$\s \gg \r$ (Case 1), the subspace $\H_1$ becomes zero; in this case, 
the first rows and columns in \eqref{eqn:blockMatrix} should be ignored.
Likewise, when $\r>0$, the subspace $\H_3$ becomes zero; in this case, 
the third rows and columns in \eqref{eqn:blockMatrix} should be ignored.

\begin{figure}[t] 
	\begin{centering}
	\includegraphics[scale=0.3]{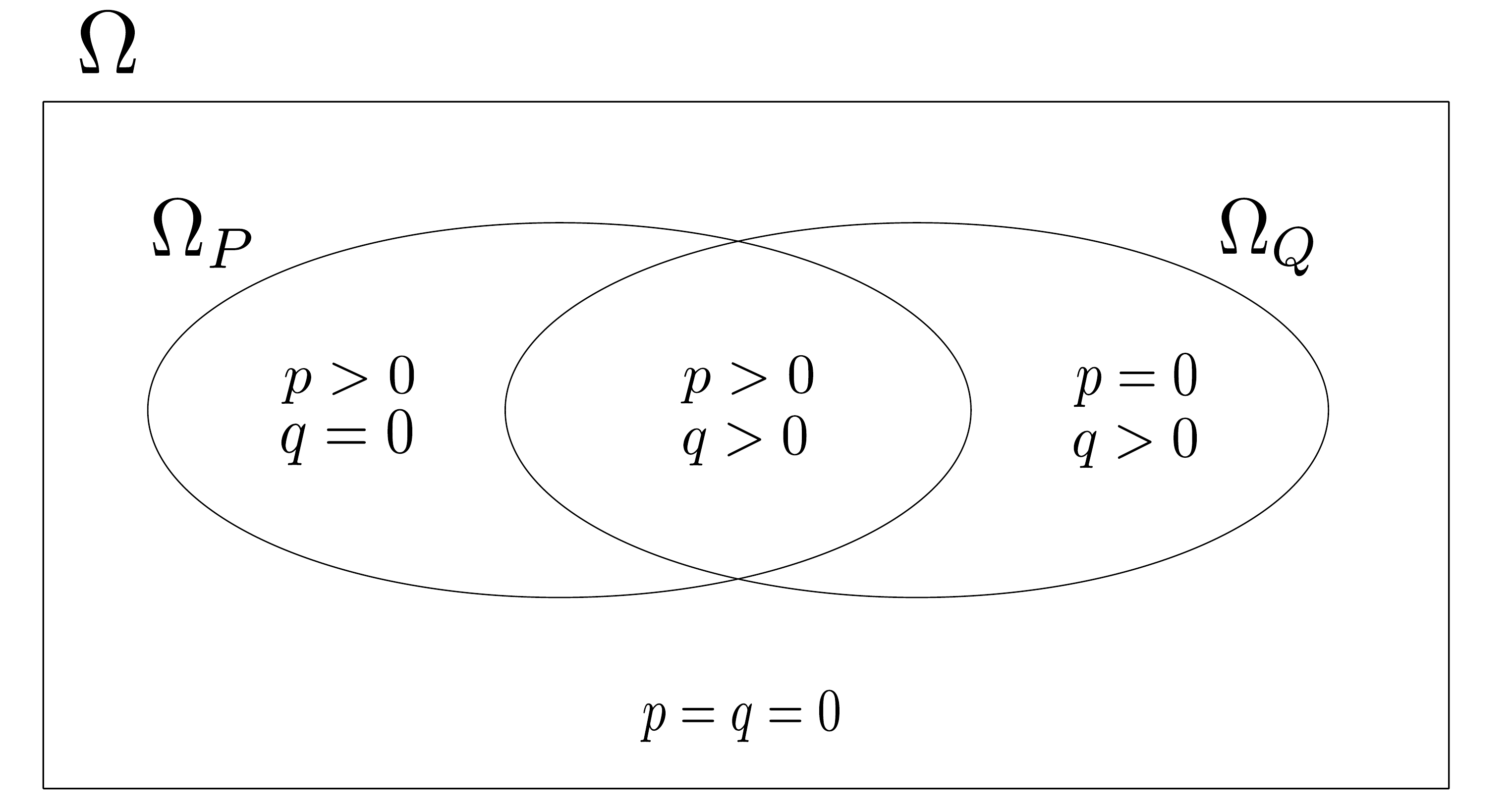}
	\par
	\end{centering}
	\caption{
	Schematic diagram of support sets of measures $P$ and $Q$ on a classical measure space 
	$(\O,\F,\m)$ having densities $p$ and $q$, respectively. 
	Here $\O_P:=\{\o\in\O\,|\, p(\o)>0\}$ and $\O_Q:=\{\o\in\O\,|\, q(\o)>0\}$. 
	The induced measures $Q^{ac}(A):=Q(A\cap \{p>0\})$ and $Q^\perp(A):=Q(A\cap \{p=0\})$ 
	give the Lebesgue decomposition $Q=Q^{ac}+Q^\perp$ with respect to $P$, 
	in which $Q^{ac}\ll P$ and $Q^\perp\perp P$, (cf. \cite[Chapter 6]{Vaart}).
	\label{fig:support}}
\end{figure}

There is an obvious similarity between the block matrix structure in \eqref{eqn:blockMatrix} 
and the diagram illustrated in Fig.~\ref{fig:support} that displays the support sets of two measures $P$ and $Q$ on a classical measure space 
$(\O,\F,\m)$ having densities $p$ and $q$, respectively. 
However,  it should be warned that
\[
 \H_1':=\supp\r\cap \ker\s,\qquad \H_2':=\supp\r\cap\supp\s
\]
are different from $\H_1$ and $\H_2$, respectively. 
This is most easily seen by considering the case when both $\r$ and $\s$ are pure states: 
for pure states $\r=\ket{\psi}\bra{\psi}$ and $\s=\ket{\xi}\bra{\xi}$, we see that $\H_2\neq \{0\}$ if and only if $\braket{\xi}{\psi}\neq 0$, (cf. \cite[Example 2.3]{YFG}), whereas $\H_2'\neq\{0\}$ if and only if $\r=\s$.

Let us rewrite $\s$ in the form
\[
\sigma
=E^*
\begin{pmatrix}
0 & 0 & 0\\
0 & \s_0 & 0 \\
0 & 0 & \beta-\alpha^* \sigma_0^{-1} \alpha
\end{pmatrix}
E,
\]
where
\[
E:=
\begin{pmatrix}
I & 0 & 0\\
0 & I & \s_0^{-1} \a \\
0 & 0 & I
\end{pmatrix}.
\]
Since $E$ is invertible and $\s\ge 0$, we see that
\[
\beta-\alpha^* \sigma_0^{-1} \alpha\ge 0.
\]
Now let 
\[
\s^{ac}
:=E^*
\begin{pmatrix} 0 & 0 & 0\\ 0 & \s_0 & 0 \\ 0 & 0 & 0\end{pmatrix}
E
=\begin{pmatrix} 0 & 0 & 0  \\ 0 & \s_0 & \a \\ 0 & \a^* & \a^* \s_0^{-1} \a \end{pmatrix}
\]
and let
\[
\s^\perp:=
E^*
\begin{pmatrix} 0 & 0 & 0 \\ 0 & 0 & 0\\ 0 & 0 & \b-\a^* \s_0^{-1} \a \end{pmatrix}
E
=\begin{pmatrix} 0 & 0 & 0 \\ 0 & 0 & 0\\ 0 & 0 & \b-\a^* \s_0^{-1} \a \end{pmatrix}.
\]
Then it is shown that $\s^{ac}\ll\r$ and $\s^\perp \perp \r$.
In fact, the latter is obvious from Lemma \ref{lem:1}. 
To prove the former, let
\[
R:=E^* \begin{pmatrix} 0 & 0 & 0 \\ 0 & \s_0 \# \r_0^{-1} & 0 \\ 0 & 0 & 0 \end{pmatrix} E.
\]
Then $R$ is a positive operator satisfying
\begin{eqnarray*}
R\r R
&= &
E^*
\begin{pmatrix} 0 & 0 & 0 \\ 0 & \s_0 \# \r_0^{-1} & 0 \\ 0 & 0 & 0 \end{pmatrix}
\begin{pmatrix} \r_2 & \r_1 & 0\\ \r_1^* & \r_0 & 0 \\ 0 & 0 & 0 \end{pmatrix}
\begin{pmatrix} 0 & 0 & 0 \\ 0 & \s_0 \# \r_0^{-1} & 0 \\ 0 & 0 & 0 \end{pmatrix}
E \\
&= &
E^*
\begin{pmatrix} 0 & 0 & 0\\ 0 & \s_0 & 0 \\ 0 & 0 & 0\end{pmatrix}
E
=
\s^{ac}.
\end{eqnarray*}
It then follows from Lemma \ref{lem:2} that $\s^{ac}\ll\r$.

In summary, given $\r$ and $\s$ that satisfy $\s\not\perp\r$, let 
\begin{equation}\label{eqn:simultaneousBlock}
\rho=\begin{pmatrix} \r_2 & \r_1 & 0\\ \r_1^* & \r_0 & 0 \\ 0 & 0 & 0 \end{pmatrix}, 
\qquad
\sigma=\begin{pmatrix} 0 & 0 & 0 \\ 0 & \s_0 & \a \\ 0 & \a^* & \beta \end{pmatrix}
\end{equation}
be their simultaneous block matrix representations relative to the aforementioned direct sum decomposition $\H=\H_1\oplus\H_2\oplus\H_3$.
Then
\begin{equation}\label{eqn:LebDec2}
\s^{ac}
=\begin{pmatrix} 0 & 0 & 0  \\ 0 & \s_0 & \a \\ 0 & \a^* & \a^* \s_0^{-1} \a \end{pmatrix},
\qquad
\s^\perp
=\begin{pmatrix} 0 & 0 & 0 \\ 0 & 0 & 0\\ 0 & 0 & \b-\a^* \s_0^{-1} \a \end{pmatrix}
\end{equation}
give the following decomposition:
\begin{equation}\label{eqn:LebesgueDecomp}
 \s=\s^{ac}+\s^\perp
 \qquad (\s^{ac}\ll \r,\;\s^\perp \perp \r)
\end{equation}
with respect to $\r$.

As in the previous subsection, we may call \eqref{eqn:LebesgueDecomp} a {\em quantum Lebesgue decomposition} for the following reasons. 
Firstly, 
although the simultaneous block representation \eqref{eqn:simultaneousBlock} has arbitrariness of unitary transformations of the form $U_1\oplus U_2\oplus U_3$,
the operators $\s^{ac}$ and $\s^\perp$ are well-defined 
because the matrices \eqref{eqn:LebDec2} are covariant under those unitary transformations.
Secondly, 
the decomposition \eqref{eqn:LebesgueDecomp} is unique, as the following lemma asserts. 

\begin{lem}\label{lem:uniqueness2}
Given $\r$ and $\s$ with $\s\not\perp\r$, 
the decomposition 
\[ \s=\s^{ac}+\s^\perp\qquad (\s^{ac} \ll \r,\;\s^\perp \perp \r) \]
is uniquely given by \eqref{eqn:LebDec2}.
\end{lem}

\begin{proof}
We show that the decomposition
\begin{equation}\label{eqn:uniqueness2}
 \s=R\r R+\t\qquad (R\ge 0,\,\t\ge 0,\; \t \perp \r)
\end{equation}
is unique.
Because of Lemma \ref{lem:uniqueness1}, it suffices to treat the case when $\s\not\gg\r$, that is, when $\H_1\neq \{0\}$. 

Let $\r$ and $\s$ be represented as \eqref{eqn:simultaneousBlock}. 
It then follows from \eqref{eqn:uniqueness2} that, for any $x\in \H_1$, 
\[
 0=\bracket{x}{\s x}\ge \bracket{x}{R \r Rx}=\bracket{Rx}{\r Rx}.
\]
This implies that $Rx\in\ker\r\,(=\H_3)$: in particular, 
$\bracket{x}{Rx}=0$, so that 
the $(1,1)$th block of $R$ is zero. 
This fact, combined with the positivity of $R$, entails that $R$ must have the form 
\[
R=
\begin{pmatrix} 
0 & 0 & 0 \\ 
0 & R_0 & R_1 \\ 
0 & R_1^* & R_2  
\end{pmatrix}.
\]
Consequently, the problem is reduced to finding the decomposition
\begin{equation}\label{eqn:uniqueness2-1}
 \hat\s=\hat R \hat\r \hat R+\hat\t\qquad (\hat R\ge 0,\,\hat\t\ge 0,\; \hat\t \perp \hat\r),
\end{equation}
where
\[
\hat\r=
\begin{pmatrix} 
\r_0 & 0 \\ 
0 & 0  
\end{pmatrix},
\qquad 
\hat\s=
\begin{pmatrix} 
\s_0 & \a \\ 
\a^* & \b  
\end{pmatrix},
\qquad 
\hat R=
\begin{pmatrix} 
R_0 & R_1 \\ 
R_1^* & R_2 
\end{pmatrix}.
\]
Since $\hat{\r} \ll \hat\s$, the uniqueness of the decomposition \eqref{eqn:uniqueness2-1} immediately follows from Lemma \ref{lem:uniqueness1}. 
This completes the proof. 
\end{proof}

Now that a quantum Lebesgue decomposition is established, we shall call the operator $R$ satisfying  \eqref{eqn:uniqueness2} the {\em square-root likelihood ratio} of $\s$ relative to $\r$, and shall denote it as $\ratio{\s}{\r}$.

\begin{rem}\label{rem:eqrtLikelihood}
The square-root likelihood ratio $R=\ratio{\s}{\r}$ is explicitly written as
\begin{equation}\label{eqn:remark}
R=\sqrt{\s} \left(\sqrt{\sqrt{\s} \r \sqrt{\s}} \right)^+ \sqrt{\s} +\ga,
\end{equation}
where $A^+$ denotes the generalized inverse of an operator $A$, and $\ga$ is an arbitrary positive operator that is singular with respect to $\rho$. 
The proof is given in Appendix \ref{sec:proofs}.
\end{rem}

\section{Contiguity}\label{sec:contiguity}

As we have seen in Introduction, the asymptotic version of absolute continuity called the contiguity played an important role in classical statistics \cite{{LeCam:1966},{LeCam:Book},{Vaart}}. 
In this section, we extend it to the quantum domain. 
There are several equivalent characterizations of the contiguity.
Among others, the following characterization is particularly relevant to our purpose because it makes no use of the notion of measurable sets that are characteristic of classical measure theory.
Let $P^{(n)}$ and $Q^{(n)}$ be sequences of probability measures on measurable spaces $(\Omega^{(n)}, \F^{(n)})$.
Then $Q^{(n)}$ is contiguous with respect to $P^{(n)}$
if and only if the sequence 
${dQ^{(n)}}/{dP^{(n)}}$
of likelihood ratios is uniformly integrable under $P^{(n)}$, and 
$\lim_{n\to\infty}  E_{P^{(n)}}\left[{dQ^{(n)}}/{dP^{(n)}} \right]=1$, (cf. \cite[Lemma V.1.10]{JacodShiryaev}). 

Let $\H^{(n)}$ be a sequence of finite dimensional Hilbert spaces, and let $\rho^{(n)}$ and $\sigma^{(n)}$ be quantum states on $\H^{(n)}$.
Further, let $R^{(n)}$ be (a version of) the square-root likelihood ratio $\ratio{\s^{(n)}}{\r^{(n)}}$.
Motivated by the above consideration, one may envisage that the sequence $\s^{(n)}$ could be designated as ``contiguous'' with respect to $\r^{(n)}$ if
\begin{itemize}
\item[(i)] $\displaystyle \lim_{n\to\infty} \Tr \rho^{(n)} R^{{(n)}^2}=1$, and
\item[(ii)] the sequence $R^{{(n)}^2}$ is uniformly integrable under $\r^{(n)}$; that is, for any $\varepsilon>0$ there exist an $M>0$ such that
\begin{equation*}
 \sup_{n} \Tr \r^{(n)} 
 R^{{(n)}^2} 
 \left( I-\mathbbm{1}_M(R^{(n)}) \right)
  < \varepsilon. 
\end{equation*}
Here, $ \mathbbm{1}_M$ is the truncation function: 
\[
 \mathbbm{1}_M(x)
 =\begin{cases}
	1,&  \text{if $|x|\leq M$}\\
	0, &  \text{otherwise}.
 \end{cases}
\]
In other words,  the operator $\mathbbm{1}_M(R^{(n)})$ is the orthogonal projection onto the subspace of $\H^{(n)}$ spanned by the eigenvectors of $R^{(n)}$ corresponding to the eigenvalues less than or equal to $M$. 
\end{itemize}

However, such a naive definition fails, as the following example demonstrates. 

\begin{example}\label{ex:pseudoLikelihood}
Let 
\[
\rho^{(n)}=\frac{1}{2n^3}\begin{pmatrix} 2n^3-1 & 0 \\ 0 & 1 \end{pmatrix},\qquad
\sigma^{(n)}=\frac{1}{2(n^2+n+1)}\begin{pmatrix}n^2 & n^2+1 \\ n^2+1 & n^2+2n+2\end{pmatrix}
\]
be sequences of faithful states on a fixed Hilbert space $\H^{(n)}=\C^2$.
For all $n\in\N$, they are mutually absolutely continuous. 
Moreover, the limiting states
\[
\rho^{(\infty)}=\begin{pmatrix} 1 & 0 \\ 0 & 0 \end{pmatrix},\qquad
\sigma^{(\infty)}=\frac{1}{2}\begin{pmatrix}  1 & 1 \\ 1 & 1 \end{pmatrix}
\]
are also mutually absolutely continuous since they are non-orthogonal pure states. 
Therefore, one would expect that $\rho^{(n)}$ and $\sigma^{(n)}$ should be contiguous.
However, this does not follow from the above naive definition.
In fact, the square-root likelihood ratio $R^{(n)}=\ratio{\s^{(n)}}{\r^{(n)}}$ is uniquely given by
\[
R^{(n)}=\frac{n}{\sqrt{2(n^2+n+1)}}\begin{pmatrix}1 & 1 \\ 1 & 2n+1 \end{pmatrix}.
\]
Therefore, for any $M> {1}/{\sqrt{2}}$,
\[
 \lim_{n\to\infty}\mathbbm{1}_M(R^{(n)})=\begin{pmatrix} 1 & 0 \\ 0 & 0 \end{pmatrix},
\]
and
\[
 \lim_{n\to\infty} \Tr \r^{(n)} R^{{(n)}^2} \left( I-\mathbbm{1}_M(R^{(n)}) \right)
 =\Tr \sigma^{(\infty)} \begin{pmatrix} 0 & 0 \\ 0 & 1 \end{pmatrix}
 =\frac{1}{2}. 
\]
Namely, $R^{{(n)}^2}$ is not uniformly integrable under $\r^{(n)}$.
	
The above strange phenomenon stems from the fact that the $(2,2)$th entry of the square-root likelihood ratio $R^{(n)}$ diverges as $n\to\infty$, 
although this entry is asymptotically inessential in that it corresponds to the singular part of the limiting reference state $\rho^{(\infty)}$. 
In other words, this divergence might be illusory in discussing the asymptotic behaviour. 
This observation may lead us to a ``modified'' positive operator 
\[
 \overline{R}^{(n)}
 =\frac{n}{\sqrt{2(n^2+n+1)}}\begin{pmatrix}1 & 1 \\ 1 & 1 \end{pmatrix}
\]
which would contain essential information about asymptotic relationship between $\r^{(n)}$ and $\s^{(n)}$.
In fact, 
\[
\overline{R}^{(n)}\rho^{(n)}\overline{R}^{(n)}
=\frac{1}{2(n^2+n+1)}\begin{pmatrix}n^2 & n^2 \\ n^2 & n^2 \end{pmatrix} 
\]
approaches $\sigma^{(\infty)}$ as $n\to \infty$, and the sequence $\overline{R}^{{(n)}^2}$ is uniformly integrable under $\rho^{(n)}$.
\end{example}

In order to formulate the idea presented in Example \ref{ex:pseudoLikelihood}, we introduce a class of modifications that is asymptotically negligible. 
We say a sequence $O^{(n)}$ of observables is {\em infinitesimal in $L^2$} (or simply {\em $L^2$-infinitesimal}) under $\rho^{(n)}$, denoted $O^{(n)}=o_{L^2}(\r^{(n)})$,  if 
\[
 \lim_{n\to\infty}\Tr\rho^{(n)}O^{(n)^{2}}=0.
\]
It is easily verified that in Example \ref{ex:pseudoLikelihood}, the operator 
$O^{(n)}:= \overline{R}^{(n)}-{R}^{(n)}$ is $L^2$-infinitesimal under $\rho^{(n)}$. 

Now we introduce a quantum extension of the contiguity. 

\begin{defn}\label{def:qcontiguity}
Let $\H^{(n)}$ be a sequence of finite dimensional Hilbert spaces, and let $\rho^{(n)}$ and $\sigma^{(n)}$ be quantum states on $\H^{(n)}$.
Further, let $R^{(n)}$ be (a version of) the square-root likelihood ratio $\ratio{\s^{(n)}}{\r^{(n)}}$.
The sequence $\s^{(n)}$ is {\em contiguous} with respect to the sequence $\r^{(n)}$, denoted
$\sigma^{(n)}\vartriangleleft\rho^{(n)}$, if 
\begin{itemize}
\item[(i)] $\lim_{n\to\infty} \Tr \rho^{(n)} R^{{(n)}^2}=1$, and 
\item[(ii)] there is an $L^2$-infinitesimal sequence $O^{(n)}$ of observables, each defined on $\H^{(n)}$, such that $\overline{R}^{(n)}:=R^{(n)}+O^{(n)}$ is positive and $\overline{R}^{{(n)}^2}$ is uniformly integrable under $\r^{(n)}$.
\end{itemize}
We also use the notation $\sigma^{(n)}\vartriangleleft_{O^{(n)}}\rho^{(n)}$ when $O^{(n)}$ needs to be specified.
\end{defn}

Several remarks are in order. 
Firstly, the above definition is independent of the choice of the square-root likelihood ratio $R^{(n)}$, since its arbitrariness (see Remark \ref{rem:eqrtLikelihood}) does not affect condition (i), and is absorbed into the $L^2$-infinitesimal modification $O^{(n)}$ in condition (ii). 
Secondly,
condition (i) and the uniform integrability in (ii) can be merged into a single condition
\[
\lim_{M\to\infty}\liminf_{n\to\infty}\Tr\rho^{(n)} \mathbbm{1}_{M}\left(\overline{R}^{(n)}\right)\overline{R}^{(n)^{2}}=1
\]
or
\[
\lim_{M\to\infty}\liminf_{n\to\infty}
\Tr\s^{(n)^{ac}}
\mathbbm{1}_{M}\left( \overline{R}^{(n)} \right)=1.
\]
Here, $\s^{(n)^{ac}}=R^{(n)} \r^{(n)} R^{(n)}$ is the absolutely continuous part of $\s^{(n)}$ 
with respect to $\r^{(n)}$. 
Thirdly, the definition is unitarily covariant, in that 
\[
\sigma^{(n)} \vartriangleleft_{O^{(n)}} \rho^{(n)}
\quad \mbox{if and only if} \quad
U^{(n)} \sigma^{(n)} U^{(n)*} \vartriangleleft_{U^{(n)} O^{(n)} U^{(n)*}} 
U^{(n)} \rho^{(n)} U^{(n)*},
\]
where $U^{(n)}$ is an arbitrary unitary operator on $\H^{(n)}$.
This fact could be useful in representing a state in a matrix form.
Fourthly, the positivity of $\overline{R}^{(n)}$ can be replaced with an asymptotic positivity; 
that is, the negative part of $\overline{R}^{(n)}$ is $L^2$-infinitesimal under $\rho^{(n)}$. 
However, the positivity of $\overline{R}^{(n)}$, whether asymptotically or not, is indispensable as the following example illustrates. 

\begin{example}\label{exa:positivity}
Let
\[
\rho^{(n)}=\begin{pmatrix}1 & 0 \\ 0 & 0 \end{pmatrix},\qquad
\sigma^{(n)}=\frac{1}{1+n^{2}}\begin{pmatrix}1 & n \\ n & n^{2} \end{pmatrix}
\]
be sequences of pure states on $\H^{(n)}=\C^2$.  
The square-root likelihood ratio $\ratio{\s^{(n)}}{\r^{(n)}}$ is given by
\[
R^{(n)}=\frac{1}{\sqrt{1+n^{2}}}\begin{pmatrix}1 & n \\ n & n^{2}+\gamma \end{pmatrix},
\]
where $\gamma$ is an arbitrary nonnegative number.
Now let
\[
O^{(n)}=\frac{1}{\sqrt{1+n^{2}}}\begin{pmatrix}0 & 0 \\ 0 & -n^{2} \end{pmatrix}
\]
and let $\overline{R}^{(n)}=R^{(n)}+O^{(n)}$. 
Then $\overline{R}^{(n)}$ is uniformly bounded, and conditions (i) and (ii) in Definition \ref{def:qcontiguity}, except the positivity of $\overline{R}^{(n)}$, are fulfilled. However, the limiting states
\[
\rho^{(\infty)}=\begin{pmatrix}1 & 0 \\ 0 & 0 \end{pmatrix},\qquad
\sigma^{(\infty)}=\begin{pmatrix} 0 & 0 \\ 0 & 1 \end{pmatrix}
\]
are mutually singular. 
\end{example}

The validity of Definition \ref{def:qcontiguity} is demonstrated by the following

\begin{thm}\label{thm:contiguity_ac}
Let $\rho^{(n)}$ and $\sigma^{(n)}$ be sequences of quantum states on a fixed finite dimensional Hilbert space $\H$, and suppose that they have the limiting states $\lim_{n\to\infty}\rho^{(n)}=\rho^{(\infty)}$ and $\lim_{n\to\infty}\sigma^{(n)}=\sigma^{(\infty)}$. 
Then $\sigma^{(n)}\vartriangleleft\rho^{(n)}$ if and only if $\sigma^{(\infty)}\ll\rho^{(\infty)}$.
\end{thm}

When the reference states $\rho^{(n)}$ are pure, there is a simple criterion for the contiguity. 

\begin{thm}\label{thm:contiguity_pure}
Let $\H^{(n)}$ be a sequence of finite dimensional Hilbert spaces, and let $\rho^{(n)}$ and $\sigma^{(n)}$ be quantum states on $\H^{(n)}$.
Suppose that $\rho^{(n)}$ is pure for all $n\in\N$. 
Then $\sigma^{(n)}\vartriangleleft\rho^{(n)}$ if and only if $\lim_{n\to\infty} \Tr\rho^{(n)}R^{(n)^{2}}=1$ and $\liminf_{n\to\infty}\Tr\rho^{(n)}\sigma^{(n)}>0$, where $R^{(n)}$ is 
(a version of) the square-root likelihood ratio $\ratio{\sigma^{(n)}}{\rho^{(n)}}$.
\end{thm}

The proofs of Theorems \ref{thm:contiguity_ac} and \ref{thm:contiguity_pure} are lengthy, and are deferred to Appendix \ref{sec:proofs}.

\section{Convergence in distribution}\label{sec:convergenceInLaw}

In this section we introduce a quantum extension of the notion of convergence in distribution in terms of the ``quasi-characteristic'' function \cite{{qclt},{YFG}}.
This mode of convergence turns out to be useful in asymptotic theory of quantum statistics. 
For a brief account of quantum Gaussian states, see Appendix \ref{app:q-Gaussian}.

\begin{defn} \label{def:conv_in_qlaw}
For each $n\in \N$, let $\rho^{(n)}$ be a quantum state and ${X}^{(n)}=\left(X_{1}^{(n)},\dots,X_{d}^{(n)}\right)$ be a list of observables on a finite dimensional Hilbert space $\H^{(n)}$. 
Further, let $\phi$ be a normal state (represented by a linear functional) and ${X}^{(\infty)}=\left(X_{1}^{(\infty)},\dots,X_{d}^{(\infty)}\right)$ be a list of observables on a possibly infinite dimensional Hilbert space $\H^{(\infty)}$ such that $\xi^i X_i^{(\infty)}$ is densely defined for every $\x=(\xi^i)\in\R^d$. 
We say the sequence 
$\left({X}^{(n)},\rho^{(n)}\right)$ {\em converges in distribution} to $\left({X}^{(\infty)},\phi \right)$, in symbols
\[ 
({X}^{(n)},\rho^{(n)}) \rightsquigarrow \left({X}^{(\infty)}, \phi \right),
\]
if
\[
\lim_{n\to\infty}\Tr\rho^{(n)}\left(\prod_{t=1}^r e^{\i\xi_{t}^{i}X_{i}^{(n)}}\right)
=
\phi \left(\prod_{t=1}^r e^{\i\xi_{t}^{i}X_{i}^{(\infty)}}\right)
\]
holds for any $r\in\N$ and subset $\{\xi_{t}\}_{t=1}^{r}$ of $\R^d$.
When the limiting state $\phi$ is a quantum Gaussian state, in that $\left({X}^{(\infty)},\phi \right)\sim N(h,J)$, we also use the abridged notation
\[
 {X}^{(n)} \convd{\rho^{(n)}}  N(h,J),
\]
in accordance with the convention in classical statistics.
\end{defn}

A slight generalization is the following mode of convergence, which plays an essential role in the present paper. 

\begin{defn}\label{def:conv_in_qlaw_sandwich}
In addition to the setting for Definition \ref{def:conv_in_qlaw}, 
let $Y^{(n)}$ and $Y^{(\infty)}$ be observables on $\H^{(n)}$ and $\H^{(\infty)}$, respectively, 
with $Y^{(\infty)}$ being densely defined.
If
\[
\lim_{n\to\infty}\Tr\rho^{(n)}e^{\i\eta_{1}Y^{(n)}}\left\{ \prod_{t=1}^{r}e^{\i\xi_{t}^{i}X_{i}^{(n)}}\right\} e^{\i\eta_{2}Y^{(n)}}
=
\phi \left(e^{\i\eta_{1}Y^{(\infty)}}\left\{ \prod_{t=1}^{r}e^{\i\xi_{t}^{i}X_{i}^{(\infty)}}\right\} e^{\i\eta_{2}Y^{(\infty)}}\right)
\]
holds for any $r\in\N$, subset $\{\xi_{t}\}_{t=1}^{r}$ of $\R^{d}$, and $\eta_1,\eta_2\in\R$, 
then we denote 
\begin{equation*}
\left( \sand{Y^{(n)}}{{X}^{(n)}}, \r^{(n)} \right)
\rightsquigarrow
\left( \sand{Y^{(\infty)}}{{X}^{(\infty)}}, \phi \right)
\end{equation*}
or  
\begin{equation*}
\sand{Y^{(n)}}{{X}^{(n)}}_{\r^{(n)}}
\rightsquigarrow
\sand{Y^{(\infty)}}{{X}^{(\infty)}}_{\phi}.
\end{equation*}
We shall call this type of convergence a {\em sandwiched convergence in distribution} to emphasize that the observables $Y^{(n)}$ and $Y^{(\infty)}$ that appear at both ends of the quasi-characteristic function play special roles.
\end{defn}

The sandwiched convergence in distribution will be used in conjunction with the following form of the quantum L\'evy-Cram\'er continuity theorem.

\begin{lem}\label{lem:qLevyCramerSand}
Let $(X^{(n)}, Y^{(n)}, \rho^{(n)})$ and $(X^{(\infty)}, Y^{(\infty)}, \phi)$ be as in Definition \ref{def:conv_in_qlaw_sandwich}. 
If 
\[
\sand{Y^{(n)}}{{X}^{(n)}}_{\r^{(n)}}
\rightsquigarrow
\sand{Y^{(\infty)}}{{X}^{(\infty)}}_{\phi},
\]
then
\begin{equation}\label{eq:bounded_conv_beside}
\lim_{n\to\infty}\Tr\rho^{(n)}g_{1}(Y^{(n)})\left\{ \prod_{t=1}^{r}f_{t}( \xi^i_t X_{i}^{(n)})\right\} g_{2}(Y^{(n)})
=
\phi\left(g_{1}(Y^{(\infty)})\left\{ \prod_{t=1}^{r}f_{t}(\xi^i_t X_{i}^{(\infty)})\right\} g_{2}(Y^{(\infty)})\right)
\end{equation}
holds for any $r\in\N$, subset $\{\xi_{t}\}_{t=1}^{r}$ of $\R^{d}$, bounded continuous functions $f_{1},\dots,f_{r}$, and bounded Borel functions $g_{1},g_{2}$ on $\R$ 
such that the set $\D(g_{i})$ of discontinuity points of $g_{i}$ has $\mu$-measure zero for $i=1,2$, where $\mu$ is the classical probability measure on $\R$ having the characteristic function $\varphi_\mu(\eta):=\phi ( e^{\i \eta Y^{(\infty)}} )$.
\end{lem}

\begin{proof}
Let $s:=r+2$, and let $J$ be an arbitrary natural number between $1$ and $s-1$ (say $J=1$).
Then the list of observables 
\[ 
{Z}^{(n)}
=(Z_{1}^{(n)},\dots,Z_{s}^{(n)})
:=(Y^{(n)}, \xi^i_1 X_i^{(n)},\dots, \xi^i_r X_i^{(n)} ,Y^{(n)})
\]
fulfils conditions 
\eqref{eq:quasi-conv1}, \eqref{eq:quasi-conv2}, and  \eqref{eq:quasi-conv3} 
in the quantum L\'evy-Cram\'er continuity Theorem \ref{thm:qLevyCramer} cited in Appendix \ref{app:q-LevyCramer}. 
Furthermore, the functions $g_1$ and $g_2$ satisfy condition \eqref{eq:levycramer_measure} in the theorem. 
Thus the claim is an immediate consequence of Theorem \ref{thm:qLevyCramer}. 
\end{proof}

In classical statistics, if random variables $X^{(n)}$ converge in distribution to a random variable $X$, and random variables $O^{(n)}$ converge in $L^2$ (and hence in probability) to 0, then $X^{(n)}+O^{(n)}$ converge in distribution to $X$ \cite[Lemma 2.8]{Vaart}.
However, its obvious analogue in quantum statistics fails to be true, as the following example illustrates.

\begin{example}
Let
\[
\rho^{(n)}:=\begin{pmatrix}1 & 0\\0 & 0\end{pmatrix}, \qquad
X^{(n)}:=\begin{pmatrix}1 & n\\n & 1+n^{2}\end{pmatrix}, \qquad
O^{(n)}:=\begin{pmatrix}0 & 0\\0 & -n^{2}\end{pmatrix}. 
\]
It is not difficult to verify that
\[ 
\lim_{n\to\infty} \Tr \r^{(n)} e^{\i \xi X^{(n)}} = 1 
\]
for all $\xi\in\R$, and $O^{(n)}=o_{L^{2}}\left(\rho^{(n)}\right)$.
However 
\[ 
 \Tr \r^{(n)} e^{\i \xi (X^{(n)}+O^{(n)})} 
 =e^{ \i \xi} \cos n\xi,
\]
which has no limit as $n\to\infty$. 
\end{example}

The above example shows that an $L^2$-infinitesimal sequence of observables is not always negligible in quasi-characteristic functions.
We therefore introduce another class of infinitesimal objects pertinent to the convergence in distribution. 

\begin{defn} \label{defn:infinitesimal}
Let $\H^{(n)}$ be a sequence of finite dimensional Hilbert spaces, and let $Z^{(n)}$ and $\r^{(n)}$ be an observable and a state on $\H^{(n)}$.
We say a sequence $O^{(n)}$ of observables, each defined on $\H^{(n)}$, is {\em infinitesimal in distribution} (or simply {\em D-infinitesimal}) with respect to $(Z^{(n)}, \r^{(n)})$, denoted $O^{(n)}=o_D\left(Z^{(n)}, \rho^{(n)}\right)$,
if
\begin{equation}\label{eq:infini_p}
\lim_{n\to\infty} \Tr\rho^{(n)}\left\{ \prod_{t=1}^{r}e^{\sqrt{-1}(\xi_{t}Z^{(n)}+\eta_{t}O^{(n)})}\right\}
=\lim_{n\to\infty} \Tr\rho^{(n)}\left\{ \prod_{t=1}^{r}e^{\sqrt{-1}\xi_{t}Z^{(n)}}\right\}
\end{equation}
holds for any $r\in\N$, and subsets $\left\{ \xi_{t}\right\} _{t=1}^{r}$ and $\left\{\eta_t\right\} _{t=1}^{r}$ of $\R$.
\end{defn}

The following lemma asserts that a D-infinitesimal sequence is negligible in the sandwiched convergence.

\begin{lem}\label{thm:inifini_both_side}
If $\sand{Z^{(n)}}{{X}^{(n)}}\convd{\r^{(n)}}\sand{Z^{(\infty)}}{{X}^{(\infty)}}$ and $O^{(n)}=o_D\left(Z^{(n)}, \rho^{(n)}\right)$ then
\begin{equation*}
	\sand{Z^{(n)}+O^{(n)}}{{X}^{(n)}}
	\convd{\r^{(n)}}
	\sand{Z^{(\infty)}}{{X}^{(\infty)}}.
\end{equation*}
\end{lem}

The proof of Lemma \ref{thm:inifini_both_side} is straightforward, and is deferred to Appendix \ref{sec:proofs}.

\section{Le Cam's third Lemma}\label{sec:lecam3}

We are now ready to extend Le Cam's third lemma to the quantum domain. 
Our first result is the following abstract version of Le Cam's third lemma, a noncommutative analogue of \cite[Theorem 6.6]{Vaart}.

\begin{thm}\label{thm:qLeCam3}
Given a sequence $\H^{(n)}$ of finite dimensional Hilbert spaces, let $\rho^{(n)}$ and $\sigma^{(n)}$ be quantum states and let ${X}^{(n)}=\left(X_{1}^{(n)},\dots,X_{d}^{(n)}\right)$ be a list of observables on $\H^{(n)}$. 
Further, let $R^{(n)}$ be (a version of) the square-root likelihood ratio $\ratio{\s^{(n)}}{\r^{(n)}}$. 
Suppose that
\begin{itemize}
\item[{\rm (i)}]  
there exists an $L^2$-infinitesimal sequence $O^{(n)}$ of observables
such that $\s^{(n)}\vartriangleleft_{O^{(n)}} \r^{(n)}$, and
\item[{\rm (ii)}] 
there exist a normal state $\phi$, 
a list of observables ${X}^{(\infty)}=\left(X_{1}^{(\infty)},\dots,X_{d}^{(\infty)}\right)$, 
and a positive observable $R^{(\infty)}$
on a possibly infinite dimensional Hilbert space $\H^{(\infty)}$
such that
\[
\sand{R^{(n)} + O^{(n)}}{{X}^{(n)}}_{\r^{(n)}}
\rightsquigarrow
\sand{R^{(\infty)}}{{X}^{(\infty)}}_{\phi},
\]
\end{itemize}
Then 
\[
\left( {X}^{(n)}, \s^{(n)} \right)
\rightsquigarrow
\left( {X}^{(\infty)}, \psi \right),
\]
where $\psi$ is a normal state on $\H^{(\infty)}$ defined by
\begin{equation}\label{eqn:psi}
\psi(A):=\phi\left(R^{(\infty)} A R^{(\infty)} \right)
\end{equation}
for bounded operators $A\in\B(\H^{(\infty)})$.
\end{thm}

In order to get a better understanding of Theorem \ref{thm:qLeCam3}, we give an informal interpretation. 
Let 
\[ (\sigma^{(n)})^{ac}=R^{(n)}\rho^{(n)} R^{(n)} \]
be the absolutely continuous part of $\sigma^{(n)}$ with respect to $\rho^{(n)}$. 
Then, thanks to the contiguity (i) and the sandwiched convergence in distribution (ii), the absolutely continuous part $(\sigma^{(n)})^{ac}$ converges (in a certain sense) to a density operator
\[ \sigma^{(\infty)}:=R^{(\infty)}\rho^{(\infty)} R^{(\infty)} \]
on $\H^{(\infty)}$, where $\rho^{(\infty)}$ is the density operator of $\phi$, so that
\[ \Tr \sigma^{(\infty)} A=\Tr \left(R^{(\infty)}\rho^{(\infty)} R^{(\infty)}\right) A = \Tr \rho^{(\infty)} \left(R^{(\infty)} A R^{(\infty)}\right). \] 
Letting $\psi(A):=\Tr \sigma^{(\infty)} A$, we have \eqref{eqn:psi}.
The proof of Theorem \ref{thm:qLeCam3} is slightly complicated, and is deferred to Appendix \ref{sec:proofs}.

A crucial application of Theorem \ref{thm:qLeCam3} is the following theorem, which is a natural quantum counterpart of the standard Le Cam third lemma \cite[Example 6.7]{Vaart}

\begin{thm}[Quantum Le Cam third lemma]\label{thm:LeCam3clt}
Given a sequence $\H^{(n)}$ of finite dimensional Hilbert spaces, let $\rho^{(n)}$ and $\sigma^{(n)}$ be quantum states,
and let ${X}^{(n)}=\left(X_{1}^{(n)},\dots,X_{d}^{(n)}\right)$ be a list of observables on $\H^{(n)}$. 
Further, let $R^{(n)}$ be (a version of) the square-root likelihood ratio $\ratio{\r^{(n)}}{\s^{(n)}}$. 
Suppose that there exist a sequence $O^{(n)}=o_{L^2}(\r^{(n)})$ satisfying $R^{(n)}+O^{(n)}>0$, 
and a sequence ${\tilde O}^{(n)}=o_D(\log (R^{(n)}+O^{(n)}),\r^{(n)})$ satisfying
\begin{equation}\label{eq:joint_g}
\begin{pmatrix} X^{(n)} \\ 2 \log (R^{(n)}+O^{(n)})-{\tilde O}^{(n)}\end{pmatrix}
\convd{\;\; \r^{(n)}}
	N\left(\begin{pmatrix}\mu \\ -\frac{1}{2} s^2 \end{pmatrix},
	\begin{pmatrix}\Sigma & \kappa \\ \kappa* & s^2 \end{pmatrix}\right).
\end{equation}
Here, $\mu\in\R^d$, $s\in\R$, $\kappa\in\C^d$, and $\Sigma$ is a $d\times d$ complex Hermitian positive semidefinite matrix.
Then
\begin{equation}\label{eqn:LeCam3clt1}
\s^{(n)} \triangleleft \r^{(n)}
\end{equation}
and
\begin{equation}\label{eqn:LeCam3clt2}
X^{(n)} \convd{\;\;\s^{(n)}}  \;N(\mu+{\rm Re} (\kappa),\Sigma).
\end{equation}
\end{thm}

\begin{proof}
Let $(X_{1},\,\dots,\,X_{d},\,L)$ be the defining canonical observables of the algebra $\CCR{\Im\begin{pmatrix}\Sigma & \kappa\\ \kappa^{*} & s^2 \end{pmatrix}}$, 
and let $\phi \sim N\left(\begin{pmatrix}\mu\\ -\frac{1}{2}s^2 \end{pmatrix},\begin{pmatrix}\Sigma & \kappa\\ \kappa^{*} & s^2 \end{pmatrix}\right)$. 
Further, let $\overline{R}^{(n)}:=R^{(n)}+O^{(n)}$, and let $L^{(n)}:=2 \log (\overline{R}^{(n)})$. 
It then follows from \eqref{eq:joint_g} that 
\[
\sand{L^{(n)}-{\tilde O}^{(n)}}{X^{(n)}}_{\r^{(n)}} \rightsquigarrow \sand{L}{X}_{\phi}. 
\]
With Lemma \ref{thm:inifini_both_side}, this implies that
\begin{equation}\label{eqn:thm6.3sandwich}
\sand{L^{(n)}}{X^{(n)}}_{\r^{(n)}} \rightsquigarrow \sand{L}{X}_{\phi}. 
\end{equation}

Let us introduce a complex-valued bounded continuous function
\[
 f_{\eta}(x):=\exp\left[ \i \,\eta \left\{ \exp\left(\frac{x}{2} \right) \right\}\right]
\]
on $\R$ having a real parameter $\eta \in\R$.
It then follows from \eqref{eqn:thm6.3sandwich} and the sandwiched version of the quantum L\'evy-Cram\'er continuity theorem (Lemma \ref{lem:qLevyCramerSand}) that 
\begin{eqnarray*}
\lim_{n\to\infty}\Tr\rho^{(n)}
f_{\eta_1}(L^{(n)}) \left\{ \prod_{t=1}^{r}e^{\i\xi_{t}^{i}X_{i}^{(n)}}\right\} f_{\eta_2}(L^{(n)})
=\phi \left(f_{\eta_1}(L)\left\{ \prod_{t=1}^{r}e^{\i\xi_{t}^{i}X_{i}}\right\} f_{\eta_2}(L)\right),
\end{eqnarray*}
where $\eta_1,\eta_2\in\R$. 
This equality is rewritten as
\begin{eqnarray*}
\sand{e^{\frac{1}{2}L^{(n)}}}{X^{(n)}}_{\r^{(n)}} \rightsquigarrow \sand{e^{\frac{1}{2}L}}{X}_{\phi}, 
\end{eqnarray*}
or equivalently, 
\begin{eqnarray*}
\sand{\overline{R}^{(n)}}{X^{(n)}}_{\r^{(n)}} \rightsquigarrow \sand{e^{\frac{1}{2}L}}{X}_{\phi}. 
\end{eqnarray*}
Specifically, $\overline{R}^{(n)} \convd{\r^{(n)}} e^{\frac{1}{2}L}$, 
 and Lemma \ref{lem:qLevyCramerSand} leads to
\begin{eqnarray*}
\lim_{n\to\infty} \Tr \r^{(n)} \mathbbm{1}_{M}(\overline{R}^{(n)}) \overline{R}^{(n)^2}
=
\phi\left(\mathbbm{1}_{M}(e^{\frac{1}{2} L}) e^{L} \right)
=
E\left[\mathbbm{1}_{M}(e^{\frac{1}{2} Z}) e^Z  \right],
\end{eqnarray*}
where $Z$ is a classical random variable that obeys the normal distribution $N(-\frac{1}{2} s^2, s^2)$,  
and the right-hand side converges to $E[e^Z]=1$ as $M\to \infty$.  
This implies that $\s^{(n)} \triangleleft \r^{(n)}$, proving \eqref{eqn:LeCam3clt1}.

To prove \eqref{eqn:LeCam3clt2}, we need only evaluate the quasi-characteristic function of the state $\psi$ defined by \eqref{eqn:psi}, that is, 
\begin{eqnarray*}
\psi\left(\prod_{t=1}^{r} e^{\i\xi_{t}^{i}X_{i}}\right)
=\phi\left(e^{\frac{1}{2}L}\left\{\prod_{t=1}^{r} e^{\i\xi_{t}^{i}X_{i}}\right\}e^{\frac{1}{2}L} \right).
\end{eqnarray*}
In calculating this function, it is convenient to introduce the following enlarged vectors and matrices. 
\[
\tilde{\mu}:=\begin{pmatrix}\mu \\ -\frac{1}{2}s^2 \end{pmatrix}, \quad
\tilde{\Sigma}:=\begin{pmatrix}\Sigma & \kappa \\ \kappa^{*} & s^2 \end{pmatrix},\quad
\tilde{\xi}_{0}=\tilde{\xi}_{r+1}:=\begin{pmatrix} 0 \\ -\frac{\i}{2} \end{pmatrix},\quad
\tilde{\xi}_{t}:=\begin{pmatrix} \xi_t \\ 0 \end{pmatrix},\;(1\le t\le r). 
\]
Then by using the quasi-characteristic function \eqref{eqn:q-GaussianCharFnc} of the quantum Gaussian state $\phi$, we have
\begin{eqnarray*}
&&\psi\left(\prod_{t=1}^{r} e^{\i\xi_{t}^{i}X_{i}}\right) \\
&&\quad=
\phi\left(e^{\i \left(-\frac{\i}{2} \right)L}\left\{\prod_{t=1}^{r} e^{\i\xi_{t}^{i}X_{i}}\right\}e^{\i \left(-\frac{\i}{2} \right)L} \right) \\
&&\quad=
\exp\left[ \sum_{t=0}^{r+1}\left( \i \,\tilde\xi^i_t \tilde\mu_i
 -\frac{1}{2}\sum_{t=0}^{r+1}\tilde{\xi}_{t}^{i}\tilde{\xi}_{t}^{j}\tilde{\Sigma}_{ji}\right)
 -\sum_{t=0}^{r+1}\sum_{u=t+1}^{r+1}\tilde{\xi}_{t}^{i}\tilde{\xi}_{u}^{j}\tilde{\Sigma}_{ji}\right]\\
&&\quad= 
\exp\left[ \sum_{t=1}^{r}
\left(\i\, \xi_{t}^{i} \left( \mu_{i}+\Re(\kappa_i)\right)-\frac{1}{2}\xi_{t}^{i} \xi_{t}^{j}\Sigma_{ji}\right)
-\sum_{t=1}^{r}\sum_{s=t+1}^{r}\xi_{t}^{i}\xi_{s}^{j}\Sigma_{ji}\right].
\end{eqnarray*}
This is identical to the quasi-characteristic function of the quantum Gaussian state $N(\mu+{\rm Re} (\kappa),\Sigma)$. 
The assertion \eqref{eqn:LeCam3clt2} now follows immediately from Theorem \ref{thm:qLeCam3}.
\end{proof}

\section{Applications}\label{sec:example}

In this section we present three examples to demonstrate the validity, flexibility, and applicability of our theory.

\subsection{Contiguity without absolute continuity}

For each $n\in\N$, let us consider quantum states
\[
\rho^{(n)}
=\begin{pmatrix}\rho_{2}^{(n)} & \rho_{1}^{(n)} & 0\\
\rho_{1}^{(n)^{*}} & \rho_{0}^{(n)} & 0\\
0 & 0 & 0 \end{pmatrix},
\qquad
\sigma^{(n)}
=\begin{pmatrix}0 & 0 & 0\\
0 & \sigma_{0}^{(n)} & \sigma_{1}^{(n)}\\
0 & \sigma_{1}^{(n)^{*}} & \sigma_{2}^{(n)}
\end{pmatrix}
\]
on $\H^{(n)}\simeq \C^{2n+2}$, where
\[
\rho_{0}^{(n)}=\frac{1}{4n^{3}}
\begin{pmatrix}2n^{3}-1 & 0\\
0 & 1
\end{pmatrix},
\qquad
\sigma_{0}^{(n)}
=\frac{1-1/(2n)}{2(n^{2}+n+1)}
\begin{pmatrix}
n^2 & n^2+1\\
n^2+1 & n^2+2n+2
\end{pmatrix},
\]
\[
\r_1^{(n)^*}= 
\frac{1}{(n+1)^3}
\begin{pmatrix}
1 & \cdots & 1\\
1 & \cdots & 1
\end{pmatrix},
\qquad
\s_1^{(n)}= 
\frac{1}{(n+1)^3}
\begin{pmatrix}
1 & \cdots & 1\\
1 & \cdots & 1
\end{pmatrix},
\]
and
\[
\r_2^{(n)}= 
\frac{1}{2 n}I_{n},
\qquad
\s_2^{(n)}= 
\frac{1}{2 n^2}I_{n},
\]
with $I_n$ the $n\times n$ identity matrix. 
Note that, for all $n\in\N$, $\sigma^{(n)}$ is not absolutely continuous to $\rho^{(n)}$ because the singular part 
\[
\s^{(n)^{\perp}}
=\begin{pmatrix}
0 & 0 & 0\\
0 & 0 & 0\\
0 & 0 & \s_{2}^{(n)}-\s_{1}^{(n)^*}\s_{0}^{(n)^{-1}}\s_{1}^{(n)} 
\end{pmatrix}
\]
is nonzero.
However, $\s^{(n)}$ is ``asymptotically'' absolutely continuous to $\rho^{(n)}$ in that $\lim_{n\to\infty}\s^{(n)^{\perp}}=0$.
Furthermore, the $(2,2)$th blocks $\rho_{0}^{(n)}$ and $\sigma_{0}^{(n)}$ are identical, up to scaling, to the states studied in Example \ref{ex:pseudoLikelihood}. 
Therefore, it is expected that $\sigma^{(n)}$ would be contiguous to $\rho^{(n)}$. 
This expectation is justified by the following more general assertion.

\begin{thm}\label{thm:intersect}
For each $n\in\N$, let
\[
\rho^{(n)}
=\begin{pmatrix}\rho_{2}^{(n)} & \rho_{1}^{(n)} & 0\\
\rho_{1}^{(n)^{*}} & \rho_{0}^{(n)} & 0\\
0 & 0 & 0 \end{pmatrix},
\qquad
\sigma^{(n)}
=\begin{pmatrix}0 & 0 & 0\\
0 & \sigma_{0}^{(n)} & \sigma_{1}^{(n)}\\
0 & \sigma_{1}^{(n)^{*}} & \sigma_{2}^{(n)}
\end{pmatrix}
\]
be quantum states on a Hilbert space $\H^{(n)}$ represented by block matrices, where
\[
\begin{pmatrix}\rho_{2}^{(n)} & \rho_{1}^{(n)}\\
\rho_{1}^{(n)^{*}} & \rho_{0}^{(n)}\\
\end{pmatrix}
>0,
\qquad
\begin{pmatrix}
\sigma_{0}^{(n)} & \sigma_{1}^{(n)}\\
\sigma_{1}^{(n)^{*}} & \sigma_{2}^{(n)}
\end{pmatrix}>0.
\]
Suppose that
\[
 \liminf_{n\to\infty}\Tr\rho_{0}^{(n)}>0, \qquad
 \lim_{n\to\infty}\Tr\sigma_{0}^{(n)}=1, 
\]
and
\[
 \frac{\rho_{0}^{(n)}}{\Tr\rho_{0}^{(n)}} \vartriangleright \frac{\sigma_{0}^{(n)}}{\Tr\sigma_{0}^{(n)}}.
\]
Then we have $\rho^{(n)}\vartriangleright \sigma^{(n)}$. 
\end{thm}

The proof of Theorem \ref{thm:intersect} is deferred to Appendix \ref{sec:proofs}.

\subsection{Contiguity for tensor product states}

Let us consider tensor product states
\[
 \rho^{(n)}:=\bigotimes_{i=1}^n \rho_i,\qquad
 \sigma^{(n)}:=\bigotimes_{i=1}^n \sigma_i,
\]
where $\rho_i$ and $\sigma_i$ are quantum states on a finite dimensional Hilbert space $\H_i$. 
Suppose that $\sigma_i \ll \rho_i$ for all $i$. 
Then $\sigma^{(n)}\ll \rho^{(n)}$ for all $n\in\N$.
It is thus natural to enquire whether or not $\sigma^{(n)}$ is contiguous with respect to $\rho^{(n)}$.
The answer is given by the following

\begin{thm}\label{thm:kakutani}
Let $\rho_i$ and $\sigma_i$ be quantum states on a finite dimensional Hilbert space $\H_i$ that satisfy $\sigma_i \ll \rho_i$, and let 
\[
 \rho^{(n)}:=\bigotimes_{i=1}^n \rho_i,\qquad
 \sigma^{(n)}:=\bigotimes_{i=1}^n \sigma_i.
\]
Then $\sigma^{(n)} \vartriangleleft \rho^{(n)}$ if and only if
\begin{equation}\label{eqn:kakutaniCriterion}
  \prod_{i=1}^\infty \Tr\rho_i R_i>0, 
\end{equation}
or equivalently
\begin{equation}\label{eqn:kakutaniCriterion2}
 \sum_{i=1}^\infty (1-\Tr\rho_i R_i)<\infty,
\end{equation}
where $R_i$ is (a version of) the square-root likelihood ratio $\ratio{\s_i}{\r_i}$.
\end{thm}

The proof of Theorem \ref {thm:kakutani} is deferred to Appendix \ref{sec:proofs}.

\begin{rem}
Theorem \ref{thm:kakutani} bears obvious similarities to Kakutani's theorem for infinite product measures \cite{{Kakutani},{Williams}} and its noncommutative extension due to Bures \cite{Bures}. 
In fact, by using Remark \ref{rem:eqrtLikelihood}, conditions \eqref{eqn:kakutaniCriterion} and \eqref{eqn:kakutaniCriterion2} are rewritten as
\[
 \prod_{i=1}^\infty \Tr \sqrt{\sqrt{\sigma_i}\,\rho_i\,\sqrt{\sigma_i}}>0
 \qquad\mbox{and}\qquad
 \sum_{i=1}^\infty \left(1-\Tr \sqrt{\sqrt{\sigma_i}\,\rho_i\,\sqrt{\sigma_i}} \right)<\infty.
\]
The summand in the latter condition is identical, up to a factor of $2$, to the square of the Bures distance between $\r_i$ and $\s_i$. 
The main difference is that we are dealing with sequences of finite tensor product states rather than infinite tensor product states.
\end{rem}

Let us give a simple example that demonstrates the criterion established in Theorem \ref{thm:kakutani}.
Let
\[
\rho=\frac{1}{2}\begin{pmatrix} 1 & 0 \\ 0 & 1 \end{pmatrix},\qquad
\sigma_t=\frac{1}{4t^2+2}\begin{pmatrix} 2t^2+2t+1 & 2t \\ 2t & 2t^2-2t+1 \end{pmatrix}, 
\]
where $t$ is a parameter with $t\ge 1$, and let us consider three sequences of tensor product states: 
\[
 \rho^{(n)}:=\bigotimes_{i=1}^n \rho,
 \qquad
 \sigma^{(n)}:=\bigotimes_{i=1}^n \sigma_i,
 \qquad
 \tilde{\sigma}^{(n)}:=\bigotimes_{i=1}^n \sigma_{\sqrt{i}}. 
\]
Since $\sigma_t\to\rho$ as $t\to\infty$, it is meaningful to enquire whether or not $\sigma^{(n)}$ and $\tilde{\sigma}^{(n)}$ are contiguous to $\rho^{(n)}$. 
As a matter of fact, $\sigma^{(n)}$ is contiguous to $\rho^{(n)}$, whereas $\tilde{\sigma}^{(n)}$ is not; this is proved as follows.
The square-root likelihood ratio $R_t=\ratio{\s_t}{\r}$ is 
\[
 R_t=\frac{1}{\sqrt{4t^2+2}}\begin{pmatrix} 2t+1 & 1 \\ 1 & 2t-1 \end{pmatrix},
\]
and thus
\[
 \Tr\rho R_t=\sqrt{\frac{2t^2}{2t^2+1}}.
\]
In view of the criterion \eqref{eqn:kakutaniCriterion2}, it suffices to verify that 
\[
  \sum_{n=1}^\infty \left(1-\sqrt{\frac{2n^2}{2n^2+1}} \right)<\infty
  \qquad\mbox{and}\qquad
  \sum_{n=1}^\infty \left(1-\sqrt{\frac{2n}{2n+1}} \right)=\infty,
\]
and this is elementary. 
These results could be paraphrased by saying that the sequence $\sigma_{n}$ converges to $\r$ quickly enough for $\sigma^{(n)}$ to be contiguous with respect to $\rho^{(n)}$,
whereas the sequence $\sigma_{\sqrt{n}}$ does not.

\subsection{Local asymptotic normality}\label{sec:qLAN}

In \cite{YFG}, we formulated a direct analogue of the weak LAN in the quantum domain.
However, that formulation was not fully satisfactory because it was applicable only to quantum statistical models that comprise mutually absolutely continuous density operators.
Here we enlarge the scope of weak q-LAN to a much wider class of models by taking advantage of the quantum Lebesgue decomposition and quantum contiguity.

\begin{defn}\label{def:QLAN}
For each $n\in\N$, let $\S^{(n)}=\left\{ \rho_{\theta}^{(n)} \left|\,\theta\in\Theta\subset\R^{d}\right.\right\}$ be a $d$-dimensional quantum statistical model on a finite dimensional Hilbert space $\H^{(n)}$,
where $\Theta$ is an open set.
We say $\S^{(n)}$ is {\em locally asymptotically normal} at $\theta_{0}\in\Theta$ if 
\begin{itemize}
\item [{\rm (i)}] there exist a list $\Delta^{(n)}=\left(\Delta_{1}^{(n)},\dots,\Delta_{d}^{(n)}\right)$ of observables on each $\H^{(n)}$ that satisfies
\[
\Delta^{(n)}\convd{\;\; \rho_{\theta_{0}}^{(n)}} \;N(0,J),
\]
where $J$ is a $d\times d$ Hermitian positive semidefinite matrix with ${\rm Re}\,J>0$, and
\item [{\rm (ii)}] the square-root likelihood ratio 
$R_{h}^{(n)}=\ratio{\rho_{\theta_{0}+h/\sqrt{n}}^{(n)}}{\rho_{\theta_{0}}^{(n)}}$ is expanded in $h\in\R^{d}$ as
\[
R_{h}^{(n)}
= \exp\left\{\frac{1}{2}\left( h^{i}\Delta_{i}^{(n)}-\frac{1}{2}\left(J_{ij}h^{i}h^{j}\right)I^{(n)} 
+ o_D \left(h^i \Delta^{(n)}_i,\rho_{\theta_{0}}^{(n)}\right) \right)\right\}
- o_{L^2}\left(\rho_{\theta_{0}}^{(n)}\right),
\]
where $I^{(n)}$ is the identity operator on $\H^{(n)}$.
\end{itemize}
\end{defn}

Note that,
in contrast to the previous paper \cite{YFG},
we here define the local asymptotic normality in terms of the square-root likelihood ratio rather than the log-likelihood ratio; 
in particular, we do not assume that $\rho_\theta^{(n)}$ is mutually absolutely continuous with respect to $\rho_{\theta_0}^{(n)}$. 
Moreover, the present definition is pertinent to the setting for the quantum Le Cam third lemma (Theorem \ref{thm:LeCam3clt}).
In fact, we have the following 

\begin{cor}[Quantum Le Cam third lemma under q-LAN]
\label{cor:LeCam3qlan}
Let $\S^{(n)}$  be as in Definition \ref{def:QLAN}, and let $X^{(n)}=\left(X_{1}^{(n)},\dots,X_{d'}^{(n)}\right)$ be a list of observables on $\H^{(n)}$.
Suppose that $\S^{(n)}$ is locally asymptotically normal at $\theta_{0}\in\Theta$ and
\begin{equation}\label{eq:joint}
\begin{pmatrix}X^{(n)}\\
\Delta^{(n)}
\end{pmatrix}
\convd{\;\; \rho_{\theta_{0}}^{(n)}}
N\left(\begin{pmatrix}0\\
0
\end{pmatrix},\begin{pmatrix}\Sigma & \tau\\
\tau* & J
\end{pmatrix}\right).
\end{equation}
Here, $\Sigma$ and $J$ are Hermitian positive semidefinite matrices of size $d' \times d'$ and $d\times d$, respectively, with ${\rm Re}\,J>0$, and $\tau$ is a complex matrix of size $d'\times d$. 
Then 
\begin{equation}\label{eqn:cor7.2}
\rho_{\theta_{0}+h/\sqrt{n}}^{(n)}\triangleleft\rho_{\theta_{0}}^{(n)}
\qquad\mbox{and}\qquad
X^{(n)}\convd{\;\;\rho_{\theta_0+h/\sqrt{n}}^{(n)}} \;N(({\rm Re}\,\tau) h,\Sigma) 
\end{equation}
for all $h\in\R^{d}$. 
\end{cor}

\begin{proof}
From the definition of q-LAN, the square-root likelihood ratio is written as
\[
R_{h}^{(n)}=\exp\left\{ \frac{1}{2}\left(h^{i}\Delta_{i}^{(n)}-\frac{1}{2}J_{ij}h^{i}h^{j}I^{(n)}+{\tilde O}^{(n)}\right)\right\} -O^{(n)}
\]
where ${\tilde O}^{(n)}=o_D\left(h^{i}\Delta_{i}^{(n)},\rho_{\theta_{0}}^{(n)}\right)$
and $O^{(n)}=o_{L^{2}}\left(\rho_{\theta_{0}}^{(n)}\right)$. 
Let
\[
L^{(n)}:=2\log(R_{h}^{(n)}+O^{(n)})-{\tilde O}^{(n)}=h^{i}\Delta_{i}^{(n)}-\frac{1}{2}J_{ij}h^{i}h^{j}I^{(n)}. 
\]
Then \eqref{eq:joint} implies that
\[
\begin{pmatrix}X^{(n)}\\
L^{(n)}
\end{pmatrix}\convd{\rho_{\theta_{0}}^{(n)}}N\left(\begin{pmatrix}0\\
-\frac{1}{2}\,^{t}hJh
\end{pmatrix},\begin{pmatrix}\Sigma & \tau h\\
(\tau h)^{*} & \,^{t}hJh
\end{pmatrix}\right).
\]
Thus, \eqref{eqn:cor7.2} immediately follows from Theorem \ref{thm:LeCam3clt}.
\end{proof}

A prototype of Corollary \ref{cor:LeCam3qlan} first appeared in \cite[Theorem 2.9]{YFG} under the assumptions that each model $\S^{(n)}$ comprised mutually absolutely continuous density operators and the pairs $(\S^{(n)}, X^{(n)})$ were jointly q-LAN. 
In contrast, Corollary \ref{cor:LeCam3qlan} makes no use of such restrictive assumptions, and is a straightforward consequence of a much general result (Theorem \ref{thm:LeCam3clt}).
This is a notable achievement, demonstrating the advantages and usefulness of the present formulation based on the quantum Lebesgue decomposition and contiguity.

Now we restrict ourselves to the i.i.d case. 
In classical statistics, it is known that the i.i.d.~extension of a model $\{P_{\theta}  \,|\, \theta \in \Theta \subset \R^d \}$ on a measure space $(\O, \F,\m)$ having densities $p_\theta$ with respect to $\mu$ is LAN at $\theta_0$ if the model is {\em differentiable in quadratic mean} at $\theta_0$ \cite[p.~93]{Vaart}, that is, if there are random variables $\ell_1,\dots,\ell_d$ that satisfy
\[
\int_\O 
\left[
\sqrt{p_{\theta_0+h}} - \sqrt{p_{\theta_0}} - \frac{1}{2} h^i \ell_i \sqrt{p_{\theta_0}}
\right]^2 d\mu = o(\|h\|^2)
\]
as $ h\rightarrow 0$.
This condition is rewritten as
\begin{equation}\label{eq:quadratic2}
\int_\O
\left[
\sqrt{ \frac{p_{\theta_0+h}^{ac}}{p_{\theta_0}} }-1- \frac{1}{2} h^i \ell_i 
\right]^2 p_{\theta_0} d\mu
+\int_\O p_{\theta_0+h}^\perp d\mu
= o(\|h\|^2),
\end{equation}
where
\[
 p_{\theta_0+h}^{ac}(\o):=\left\{\array{ll} p_{\theta_0+h}(\o),& \o\in\O_0 \\ 0, & \o\notin \O_0\endarray\right.
\]
and
\[
 p_{\theta_0+h}^\perp(\o):=\left\{\array{ll} 0,& \o\in\O_0 \\ p_{\theta_0+h}(\o), & \o\notin \O_0\endarray\right.
\]
with $\O_0:=\{\o\in\O\,|\, p_{\th_0}(\o)>0\}$. 
The first term in the left-hand side of \eqref{eq:quadratic2} pertains to the differentiability of the (square-root) likelihood ratio at $h=0$, while the second term to the negligibility of the singular part. 

The quantum counterpart of this characterization is given by the following 

\begin{thm}[q-LAN for i.i.d.~models]\label{thm:iid}
Let $\left\{ \rho_{\theta}\left|\,\theta\in\Theta\subset\R^{d}\right.\right\} $
be a quantum statistical model on a finite dimensional Hilbert space $\H$,  
and suppose that, for some $\theta_0\in\Theta$, 
a version $R_{h}$ of the square-root likelihood ratio
$\ratio{\rho_{\theta_{0}+h}}{\rho_{\theta_{0}}}$
is differentiable at $h=0$, 
and the absolutely continuous part of $\rho_{\theta_{0}+h}$ with respect to $\rho_{\theta_0}$ satisfies
\begin{equation}\label{eq:oh2}
\Tr{\rho_{\theta_{0}} R_h^2}
=
1-o(\|h\|^{2}).
\end{equation}
Then
$\left\{ \rho_{\theta}^{\otimes n}\left|\,\theta\in\Theta\subset\R^{d}\right.\right\} $
is locally asymptotically normal at $\theta_{0}$, 
in that
\[
\Delta_{i}^{(n)}:=\frac{1}{\sqrt{n}}\sum_{k=1}^{n}I^{\otimes(k-1)}\otimes L_{i}\otimes I^{\otimes(n-k)},
\]
satisfies (i) and (ii) in Definition \ref{def:QLAN}. 
Here $L_{i}$ is (a version of) the $i$th symmetric logarithmic derivative at $\theta_{0}$, 
and $J=(J_{ij})$ is given by
\[
 J_{ij}:=\Tr{\rho_{\theta_{0}}L_{j}L_{i}}. 
\]

Further, given observables $\{B_i\}_{1\le i\le d'}$ on $\H$ satisfying $\Tr \r_{\theta_0} B_i=0$ for $i=1,\dots,d'$,
let $X^{(n)}=\{X_i^{(n)}\}_{1\le i\le d'}$ be observables on $\H^{\otimes n}$ defined by
\[
X_i^{(n)}:=\frac{1}{\sqrt{n}}\sum_{k=1}^n I^{\otimes(k-1)}\otimes B_i\otimes I^{\otimes (n-k)}.
\]
Then we have
\begin{equation}\label{eq:iidLeCam3}
\r^{\otimes n}_{\th_0+h/\sqrt{n}} \vartriangleleft \r^{\otimes n}_{\th_0}
\qquad \mbox{and}\qquad 
X^{(n)}
\convd{\rho_{\theta_{0}+h/\sqrt{n}}^{\otimes n}}
N(({\rm Re}\,\tau)h,\,\Sigma)
\end{equation}
for $h\in\R^{d}$, 
where $\Sigma$ is the $d' \times d'$ positive semidefinite matrix defined by $\Sigma_{ij}=\Tr\r_{\theta_0} B_jB_i$ 
and $\tau$ is the $d' \times d$ matrix defined by $\tau_{ij}=\Tr\r_{\theta_0}L_j B_i$.
\end{thm}

The proof of Theorem \ref{thm:iid} is deferred to Appendix \ref{sec:proofs}. 

Let us demonstrate the power of Theorem \ref{thm:iid}. 
First we recall the following two-dimensional spin-1/2 pure state model treated in Example 3.3 of \cite{YFG}:
\begin{eqnarray*}
\tilde{\r}_\th
&:=& e^{\frac{1}{2}(\th^1\s_1+\th^2\s_2-\psi(\th))}
\begin{pmatrix} 1 & 0\\ 0 & 0\end{pmatrix}
e^{\frac{1}{2}(\th^1\s_1+\th^2\s_2-\psi(\th))}\\
&=&\frac{1}{2} \left\{ I+\frac{\tanh{\|\th\|}}{\|\th\|} (\th^1 \s_1 + \th^2 \s_2) + \frac{1}{\cosh{\| \th \|}} \s_3 \right\},
\end{eqnarray*}
where 
\[
 \s_1=\begin{pmatrix} 0 & 1\\ 1 & 0\end{pmatrix},\qquad
 \s_2=\begin{pmatrix} 0 & -\i\\ \i & 0\end{pmatrix},\qquad
 \s_3=\begin{pmatrix} 1 & 0\\ 0 & -1\end{pmatrix}
\]
are the Pauli matrices, $\th=(\th^1,\th^2)\in\R^2$ are parameters to be estimated, and $\psi(\th):=\log\cosh\|\th\|$. 
A version of the square-root likelihood ratio $\ratio{\tilde{\r}_\th}{ \tilde{\r}_0}$ is given by
$\tilde{R}_\th=e^{\frac{1}{2}(\th^1\s_1+\th^2\s_2-\psi(\th))}$,
and is expanded in $\theta$ as
\[
\tilde{R}_\theta= I+\frac{1}{2}L_i \theta^i+o(\| \theta \|),
\]
where $L_i:=\s_i$ is a version of the $i$th SLD of the model $\tilde{\r}_\th$ at $\th=0$.
Let $X^{(n)}=(X_1^{(n)},X_2^{(n)})$ be defined by
\begin{equation}\label{eq:pureLeCamX} 
 X_i^{(n)}:=\Delta_i^{(n)}
 :=\frac{1}{\sqrt{n}}\sum_{k=1}^n I^{\otimes(k-1)}\otimes L_i\otimes I^{\otimes (n-k)}.
\end{equation}
Then it is shown that $\{\tilde{\r}_\th^{\otimes n}\}$ is locally asymptotically normal at $\th=0$, and 
\begin{equation}\label{eq:pureLeCam3} 
X^{(n)}\convd{\;\;\tilde{\r}_{h/\sqrt{n}}^{\otimes n}} \;
N(h,J),
\end{equation}
where
\[
J=[\Tr \tilde{\r}_0 L_j L_i]_{ij}=
\begin{pmatrix}
1 & -\i\\
\i & 1
\end{pmatrix}.
\]

Incidentally, let us investigate what happens when the scaling factor $1/\sqrt{n}$ of the parameter $\theta=h/\sqrt{n}$ is replaced with $1/g(n)$, where $g(n)>0$ and $\lim_{n\to\infty}g(n)=\infty$.
By direct computation, we have
\begin{eqnarray*}
\liminf_{n\to\infty}
\Tr \tilde{\r}^{\otimes n}_{0} \tilde{\r}^{\otimes n}_{h/g(n)} 
&=&
\liminf_{n\to\infty}
\left\{ \Tr \tilde{\r}_{0} \tilde{\r}_{h/g(n)} \right\}^n\\
&=&
\liminf_{n\to\infty}
\left\{ \frac{1}{2} \left(1+\frac{1}{\cosh({\| h \|}/{g(n)})} \right)  \right\}^n\\
&=&
\liminf_{n\to\infty}
\left\{1- \frac{\|h\|^2}{4 g(n)^2} + o\left(\frac{1}{g(n)^2} \right) \right\}^n \\
&=&
\liminf_{n\to\infty}
\left\{1- \frac{\|h\|^2}{4 g(n)^2} + o\left(\frac{1}{g(n)^2} \right) \right\}^{g(n)^2 \frac{n}{g(n)^2}} \\
&=&
\liminf_{n\to\infty}
e^{-\frac{\|h\|^2}{4} \frac{n}{g(n)^2}}.
\end{eqnarray*}
It then follows from Theorem \ref{thm:contiguity_pure} that 
$\tilde{\r}^{\otimes n}_{h/g(n)} \vartriangleleft \tilde{\r}^{\otimes n}_{0}$
if and only if ${n}/{g(n)^2}$ is bounded. 

Now we consider a perturbed model
\[
 \r_\th:=e^{-f(\th)} \tilde{\r}_\th+(1-e^{-f(\th)})\begin{pmatrix} 0 & 0\\ 0 & 1\end{pmatrix}, \qquad (\th\in\R^2),
\]
where $f(\th)$ is a smooth function that is positive for all $\th\neq 0$ and $f(0)=0$. 
Geometrically, this model is tangential to the Bloch sphere at the north pole $\r_0\,(=\tilde{\r}_0)$, 
and has a singularity at $\th=0$ in that the rank of the model drops there. 
Such a model was beyond the scope of our previous paper \cite{YFG}.

Since $\r_\th\ge e^{-f(\th)}\tilde{\r}_\th$, 
we see from Lemma \ref{lem:2} that $\r_\th\gg \r_0$ for all $\th$. 
It is also easily seen that the quantum Lebesgue decomposition 
$
 \r_\th=\r_\th^{ac}+\r_\th^\perp
$
with respect to $\r_0$ is given by
\[
 \r_\th^{ac}:=e^{-f(\th)}\tilde{\r}_\th,\qquad 
  \r_\th^\perp:=(1-e^{-f(\th)}) \begin{pmatrix} 0 & 0\\ 0 & 1\end{pmatrix}.
\]
Similarly, the quantum Lebesgue decomposition 
$
\r_\th^{\otimes n}=(\r_\th^{\otimes n})^{ac}+(\r_\th^{\otimes n})^\perp
$
with respect to $\r_0^{\otimes n}$ is given by
\[
(\r_\th^{\otimes n})^{ac}=(\r_\th^{ac})^{\otimes n},
\qquad 
(\r_\th^{\otimes n})^\perp=\r_\th^{\otimes n}-(\r_\th^{\otimes n})^{ac}. 
\]

For a positive sequence $g(n)$ satisfying $\lim_{n\to\infty}g(n)=\infty$, 
we have 
\[
\Tr (\r_{h/{g(n)}}^{\otimes n})^{ac}=e^{-n f(h/{g(n)})}
\]
and
\begin{eqnarray*}
\liminf_{n\to\infty}
 \Tr \r_0^{\otimes n} (\r_{h/{g(n)}}^{\otimes n})^{ac}
&=&
\liminf_{n\to\infty}
e^{-n f(h/{g(n)})}\left\{ \frac{1}{2}\left(1+\frac{1}{\cosh ({\|h\|}/{g(n)})} \right) \right\}^n\\
&=&
\liminf_{n\to\infty} e^{-n f(h/{g(n)}) -\frac{\|h\|^2}{4} \frac{n}{g(n)^2} }.
\end{eqnarray*}
It then follows from Theorem \ref{thm:contiguity_pure} that $\r_{h/{g(n)}}^{\otimes n} \triangleleft \r_{0}^{\otimes n}$
if and only if $n f(h/{g(n)})$ converges to zero and ${n}/{g(n)^2}$ is bounded.

For the standard scaling $g(n)=\sqrt{n}$, the above observation shows that
$\r_{h/\sqrt{n}}^{\otimes n} \triangleleft \r_{0}^{\otimes n}$
if and only if $f(\theta)=o(\|\theta\|^2)$.
Then the operator $R_\th:=e^{-\frac{1}{2}f(\th)}\tilde{R}_\th$, a version of the square-root likelihood ratio $\ratio{\r_\th}{\r_0}$, is expanded in $\theta$ as
\[
 R_\theta=I+\frac{1}{2}L_i \theta^i+o(\|\theta\|),
\]
where $L_i:=\s_i$ is a version of the $i$th SLD of the model $\r_\th$ at $\th=0$.
On the other hand, the singular part $\r_\th^\perp$ exhibits 
$\Tr \r_\theta^\perp = o(\| \theta \|^2)$; 
this ensures the condition \eqref{eq:oh2}.
It then follows from Theorem \ref{thm:iid} that
$\{\r^{\otimes n}_{\th} \}_{\th}$ is locally asymptotically normal at $\th=0$, and  
the sequence $X^{(n)}$ of observables defined by \eqref{eq:pureLeCamX} exhibits
\begin{equation}\label{eq:pureLeCam3_2}
X^{(n)}\convd{\;\;\r_{h/\sqrt{n}}^{\otimes n}} \;
N(h,J).
\end{equation}

In summary, as far as the observables $X^{(n)}=(X^{(n)}_1,X^{(n)}_2)$ defined by \eqref{eq:pureLeCamX} are concerned,
the i.i.d.~extension 
$\left\{\left.\r^{\otimes n}_{h/\sqrt{n}} \right|\, h\in\R^2 \right\}$
of the perturbed model $\r_\th$ around the singular point $\th=0$ 
is asymptotically similar to the quantum Gaussian shift model $\{ N(h,J) \left|\, h\in\R^2 \right.\}$ 
as shown in \eqref{eq:pureLeCam3_2}, 
and is also asymptotically similar to the i.i.d.~extension 
$\left\{\left.\tilde{\r}^{\otimes n}_{h/\sqrt{n}} \right|\, h\in\R^2 \right\}$
of the unperturbed pure state model $\tilde{\r}_\th$ around $\th=0$ as shown in \eqref{eq:pureLeCam3}.

We conclude this subsection with a short remark that, 
for any quantum statistical model that fulfils assumptions of Theorem \ref{thm:iid}, 
the Holevo bound \cite{Holevo:1982} is asymptotically achievable at $\th_0$. 
In fact, let $\{B_i\}_{1\le i\le d'}$ be a basis of the minimal 
$\D$-invariant extension of the SLD tangent space 
at $\th_0$, where $\D$ is the commutation operator \cite{Holevo:1982}. 
Then the Holevo bound for the original model $\{\r_\th\}_\th$ at $\th=\th_0$ 
coincides with that for the quantum Gaussian shift model 
$N(({\rm Re}\tau) h,\Sigma)$ at $h=0$, and hence at any $h$.  
Thus the asymptotic property
\[ 
X^{(n)}
\convd{\rho_{\theta_0+h/\sqrt{n}}^{\otimes n}}  
N(({\rm Re}\tau) h,\Sigma)
\]
enables us to construct a sequence of observables that asymptotically achieves the Holevo bound.
For a concrete construction of estimators, see the proof of \cite[Theorem 3.1]{YFG}.

\section{Concluding remarks}\label{sec:conclusion}

In the present paper, we first formulated a novel quantum Lebesgue decomposition (Lemma \ref{lem:uniqueness2}), 
and then developed a theory of quantum contiguity (Definition \ref{def:qcontiguity}).
We further studied the notion of convergence in distribution in the quantum domain, 
and proved a noncommutative extension of the L\'evy-Cram\'er continuity theorem 
under the sandwiched convergence in distribution (Lemma \ref{lem:qLevyCramerSand}). 
Combining these key results, we arrived at our main result, the quantum Le Cam third lemma (Theorems \ref{thm:qLeCam3} and \ref{thm:LeCam3clt}).
The power and usefulness of our theory were demonstrated by several examples, including a quantum contiguity version of the Kakutani dichotomy (Theorem \ref{thm:kakutani}), and enlargement of the scope of q-LAN (Corollary \ref{cor:LeCam3qlan}). 

We believe that the paper presented some notable progresses in asymptotic quantum statistics. 
Nevertheless, there are many open problems left to study in the future. 
Among others, it is not clear whether every sequence of positive operator-valued measures on a weak q-LAN model can be realized on the limiting quantum Gaussian shift model. 
In classical statistics, this question has been solved affirmatively by the representation theorem \cite{Vaart}, 
which asserts that, given a weakly convergent sequence $T^{(n)}$ of statistics on a LAN model
$\left\{p^{(n)}_{\th_0+h/\sqrt{n}} \left|\, h\in\R^d \right.\right\}$,
there exist a limiting statistics $T$ on the Gaussian shift model $\left\{ N(h,J^{-1}) \left|\, h\in\R^d \right.\right\}$
such that $T^{(n)} \convd{h} T$. 
Representation theorem is useful in proving, for example, the non-existence of an estimator that can asymptotically do better than what can be achieved in the limiting Gaussian shift model. 
Moreover, the so-called convolution theorem and local asymptotic minimax theorem, which are the standard tools in discussing asymptotic lower bounds for estimation in LAN models, immediately follows \cite{Vaart}. 
Extending the representation theorem, convolution theorem, and local asymptotic minimax theorem
to the quantum domain is one of the most important open problems.

It also remains to be investigated whether our asymptotically optimal statistical procedures for the local model indexed by the parameter $\theta_0+h/\sqrt{n}$ can be translated into useful statistical procedures for the real world case in which $\theta_0$ is unknown. 
Some authors \cite{{GillMassar}, {YangCH}} advocated two-step estimation procedures, in which one first measures a small portion of the quantum system, in number $n_1$ say, using some standard measurement scheme and constructs an initial estimate, say $\tilde\theta_1$, of the parameter.
One next applies the theory of q-LAN to compute the asymptotically optimal measurement scheme which corresponds to the situation $\theta_0=\tilde \theta_1$, 
and then proceeds to implement this measurement on the remaining $n_2\, (:=n-n_1)$ quantum systems collectively, estimating $h$ in the model $\theta=\tilde \theta_1+h/\sqrt{n_2}$. 
However such procedures are inherently limited to within the scope of weak consistency. 
Studying the strong consistency and asymptotic efficiency \cite{Fujiwara:2006} in the framework of collective quantum estimation scheme is also an important open problem.

\section*{Acknowledgments}

The present study was supported by JSPS KAKENHI Grant Numbers JP22340019, and JP17H02861.

\appendix

\section{Quantum Gaussian state}\label{app:q-Gaussian}
Given a $d\times d$ real skew-symmetric matrix $S=[S_{ij}]$, let $\CCR{S}$ denote the algebra generated by the observables $X=(X_1,\dots,X_d)$ that satisfy the following canonical commutation relations (CCR): 
\[ \frac{\i}{2}[X_i,X_j] = S_{ij} \qquad(1\leq i,j\leq d), \]
or more precisely
\[
 e^{\i X_i}e^{\i X_j}=e^{-\i S_{ij}} e^{\i (X_i+X_j)} \qquad(1\leq i,j\leq d). 
\]
A state $\phi$ on CCR($S$) is called a {\em quantum Gaussian state}, denoted $\phi\sim N(h,J)$, if the characteristic function 
${\cal F}_{\xi}\{\phi\}:=\phi(e^{\i\xi^{i}X_{i}})$ takes the form
\[ {\cal F}_{\xi}\{\phi\}=e^{\i\xi^{i}h_{i}-\frac{1}{2}\xi^{i}\xi^{j}V_{ij}} \]
where $\x=(\x^i)_{i=1}^d\in\R^d$, $h=(h_i)_{i=1}^d\in\R^d$, and $V=[V_{ij}]$ is a real symmetric matrix such that the Hermitian matrix $J:=V+\i S$ is positive semidefinite. 
When the canonical observables $X$ need to be specified, we also use the notation $(X,\phi)\sim N(h,J)$. 

When we discuss relationships between a quantum Gaussian state $\phi$ on a CCR and a state on another algebra, we need to use the {\em quasi-characteristic function} \cite{qclt}
\begin{equation}\label{eqn:q-GaussianCharFnc}
\phi\left(\prod_{t=1}^{r} e^{\i\xi_{t}^{i}X_{i}}\right)
=
\exp\left(\sum_{t=1}^{r}\left(\i\xi_{t}^{i}h_{i}-\frac{1}{2}\xi_{t}^{i}\xi_{t}^{j}J_{ji}\right)-\sum_{t=1}^{r}\sum_{u=t+1}^{r}\xi_{t}^{i}\xi_{u}^{j}J_{ji}\right)
\end{equation}
of a quantum Gaussian state, where $(X,\phi)\sim N(h,J)$ and $\{\x_t\}_{t=1}^r \subset \R^d$.  
Note that \eqref{eqn:q-GaussianCharFnc} is analytically continued to $\{\x_t\}_{t=1}^r \subset \C^d$.

\section{Quantum L\'evy-Cram\'er continuity theorem}\label{app:q-LevyCramer}

In \cite{qLevyCramer}, they derived a noncommutative version of the L\'evy-Cram\'er continuity theorem. 
Let us first cite their main result in a form consistent with the present paper.

For each $n\in \N$, let $\rho^{(n)}$ be a state (density operator) and ${Z}^{(n)}=\left(Z_{1}^{(n)},\dots,Z_{s}^{(n)}\right)$ be observables on a finite dimensional Hilbert space $\H^{(n)}$.
Further, let $\phi$ be a normal state (linear functional) 
and ${Z}^{(\infty)}=\left(Z_{1}^{(\infty)},\dots,Z_{s}^{(\infty)}\right)$ be densely defined observables on a possibly infinite dimensional Hilbert space $\H^{(\infty)}$. 
Assume that for all $m\in\N$, $\alpha=(\alpha_1,\dots,\alpha_m)\in\R^m$, and $j_1,\dots, j_m\in\{1,\dots,s\}$, one has
\begin{equation}\label{eqn:JaksicAssumtion}
\lim_{n\to\infty}\Tr\rho^{(n)}\prod_{t=1}^{m}e^{\sqrt{-1}\alpha_t Z_{j_t}^{(n)}}
=\phi\left(\prod_{t=1}^{m}e^{\sqrt{-1}\alpha_t Z_{j_t}^{(\infty)}} \right).
\end{equation}
Then it holds that
\begin{equation}\label{eq:bounded_conv}
\lim_{n\to\infty}\Tr\rho^{(n)}\prod_{i=1}^{s}f_{i}(Z_{i}^{(n)})
=\phi\left(\prod_{i=1}^{s}f_{i}(Z_{i}^{(\infty)})\right)
\end{equation}
for any bounded continuous functions $f_{1},\dots,f_{s}$ on $\R$. 
Furthermore, \eqref{eq:bounded_conv} remains true for bounded Borel functions $f_{1},\dots,f_{s}$ on $\R$ that enjoy certain measure conditions for the sets of discontinuity points (which will be stated below). 

Now observe that assumption \eqref{eqn:JaksicAssumtion} requires every finite repetition and permutation of the given observables $\{Z_i^{(\, \cdot\, )}\}_{1\le i\le s}$. 
Nevertheless, what Jak\v{s}i\'c {\em et al.} elucidated was something stronger in that their proof did not make full use of assumption \eqref{eqn:JaksicAssumtion} and is effective under certain weaker assumptions. 
In particular, the following variant, in which assumption \eqref{eqn:JaksicAssumtion} is replaced with \eqref{eq:quasi-conv1}--\eqref{eq:quasi-conv3}, plays a key role in the present paper. 

\begin{thm}\label{thm:qLevyCramer}
For $n\in\N\cup\{\infty\}$, $i\in\{1,\dots,s\}$, and $\alpha=(\alpha_1,\dots,\alpha_s)\in\R^{s}$, let 
$U_{i}^{-(n)}(\alpha)$ and $U_{i}^{+(n)}(\alpha)$ be unitary operators defined by
\[
U_{i}^{-(n)}(\alpha):=\prod_{t=1}^{i}e^{\sqrt{-1}\alpha_t Z_{t}^{(n)}}
\quad\mbox{and}\quad
U_{i}^{+(n)}(\alpha):=\prod_{t=i}^{s}e^{\sqrt{-1}\alpha_{t}Z_{t}^{(n)}},
\]
and let $U_{0}^{-(n)}(\alpha)$ and $U_{s+1}^{+(n)}(\alpha)$ be identity operators.
Assume that there is a $J\in\{0,1,\dots,s\}$ such that, for all $\alpha, \beta\in\R^s$, the following three conditions are satisfied: 
\begin{equation}\label{eq:quasi-conv1}
\lim_{n\to\infty}\Tr\rho^{(n)}U_{s}^{-(n)}(\alpha)
=\phi\left(U_{s}^{-(\infty)}(\alpha)\right),
\end{equation}
\begin{equation}\label{eq:quasi-conv2}
\lim_{n\to\infty}\Tr\rho^{(n)}U_{J}^{-(n)}(\alpha)\,U_{J}^{-(n)}(\beta)^{*}
=\phi\left(U_{J}^{-(\infty)}(\alpha)\,U_{J}^{-(\infty)}(\beta)^{*} \right),
\end{equation}
\begin{equation}\label{eq:quasi-conv3}
\lim_{n\to\infty}\Tr\rho^{(n)}U_{J+1}^{+(n)}(\alpha)^{*}\,U_{J+1}^{+(n)}(\beta)
=\phi\left(U_{J+1}^{+(\infty)}(\alpha)^{*}\,U_{J+1}^{+(\infty)}(\beta)\right).
\end{equation}
Then \eqref{eq:bounded_conv} holds for any bounded continuous functions $f_{1},\dots,f_{s}$ on $\R$. 
	
Furthermore, let $f_{1},\dots,f_{s}$ be bounded Borel functions on $\R$, and let $\D(f_{i})$ be the set of discontinuity points of $f_{i}$. 
Assume, in addition to \eqref{eq:quasi-conv1}--\eqref{eq:quasi-conv3}, that one has
\begin{equation} \label{eq:levycramer_measure}
\mu_{i}^{\alpha}(\D(f_{i}))=0
\end{equation}
for all $i\in\{1,\dots, s\}$ and $\alpha\in\R^{s}$, where $\mu_{i}^{\alpha}$ is the classical probability measure having the characteristic function 
\begin{equation} \label{eq:levycramer_characteristic}
\varphi_{\mu_{i}^{\alpha}}(\gamma)
:=\begin{cases}
\phi\left(U_{i-1}^{-(\infty)}(\a) \left(e^{\sqrt{-1}\gamma Z_{i}^{(\infty)}}\right) U_{i-1}^{-(\infty)}(\a)^*\right), 
& \text{if }i\le J\\ \\
\phi\left(U_{i+1}^{+(\infty)}(\a)^* \left(e^{\sqrt{-1}\gamma Z_{i}^{(\infty)}} \right) U_{i+1}^{+(\infty)}(\a)\right),
& \text{if }i\ge J+1.
\end{cases}
\end{equation}
Then \eqref{eq:bounded_conv} remains true.
\end{thm}

The proof of Theorem \ref{thm:qLevyCramer} is exactly the same as \cite{qLevyCramer}.
Note that when $J\in\{1,\dots, s-1\}$, the characteristic functions \eqref{eq:levycramer_characteristic} for $i=1$ and $s$ are reduced to
\[
\varphi_{\mu_{1}^{\alpha}}(\gamma)=\phi\left(e^{\sqrt{-1}\gamma Z_{1}^{(\infty)}}\right)
\quad\mbox{and}\quad
\varphi_{\mu_{s}^{\alpha}}(\gamma)=\phi\left(e^{\sqrt{-1}\gamma Z_{s}^{(\infty)}}\right).
\] 
In particular, they are independent of $\a$.
This fact is exploited in our sandwiched-type continuity theorem (Lemma \ref{lem:qLevyCramerSand}).

\section{Proofs}\label{sec:proofs}

\begin{proof}[Proof of Remark \ref{rem:eqrtLikelihood}]
Recall that $\s$ is decomposed as $\s=E^*\tilde\s E$, where
\[
E=
\begin{pmatrix}
I & 0 & 0\\
0 & I & \s_0^{-1} \a \\
0 & 0 & I
\end{pmatrix},
\qquad
\tilde\s=
\begin{pmatrix}
0 & 0 & 0\\
0 & \s_0 & 0 \\
0 & 0 & \beta-\alpha^* \sigma_0^{-1} \alpha
\end{pmatrix}.
\]
Then there is a unitary operator $U$ that satisfies
\[
 \sqrt{\tilde\s}\,E=U\sqrt{\s}, 
\]
and the operator $R$, modulo the singular part $R_2$, is given by
\begin{eqnarray*}
E^* \begin{pmatrix} 0 & 0 & 0 \\ 0 & \s_0 \# \r_0^{-1} & 0 \\ 0 & 0 & 0 \end{pmatrix} E
&=&
E^* \begin{pmatrix} 0 & 0 & 0 \\ 
0 & \sqrt{\s_0}\left(\sqrt{\sqrt{\s_0}\r_0\sqrt{\s_0}\,}\right)^{-1}\sqrt{\s_0} & 0 \\ 
0 & 0 & 0 \end{pmatrix} E \\
&=&
E^* \sqrt{\tilde\s}\left(\sqrt{\sqrt{\tilde\s}\r \sqrt{\tilde\s}\,}\right)^{+}\sqrt{\tilde\s}\, E\\
&=&
E^* \sqrt{\tilde\s}\left(\sqrt{\sqrt{\tilde\s} E\r E^*\sqrt{\tilde\s}\,}\right)^{+}\sqrt{\tilde\s}\, E\\
&=&
\sqrt{\s}\, U^*\left(\sqrt{ U\sqrt{\s} \r \sqrt{\s} U^*\,}\right)^{+} U \sqrt{\s} \\
&=&
\sqrt{\s}\, U^*\left(U\sqrt{\sqrt{\s} \r \sqrt{\s}\,}U^*\right)^{+} U \sqrt{\s} \\
&=&
\sqrt{\s} \left(\sqrt{ \sqrt{\s} \r \sqrt{\s}\,}\right)^{+} \sqrt{\s}.
\end{eqnarray*}
This proves the claim \eqref{eqn:remark}.
\end{proof}


\begin{proof}[Proof of Theorem \ref{thm:contiguity_ac}]
We first prove the `if' part. 
Due to Remark \ref{rem:eqrtLikelihood}, for each $n\in\N\cup \{\infty\}$, the operator 
\[ R^{(n)}:=\sqrt{\sigma^{(n)}}\,Q^{(n)^{+}}\sqrt{\sigma^{(n)}} \]
is a version of the square-root likelihood ratio $\ratio{\sigma^{(n)}}{\rho^{(n)}}$, where
\[ Q^{(n)}:= \sqrt{\sqrt{\sigma^{(n)}}\rho^{(n)}\sqrt{\sigma^{(n)}}\,}. \]
Let the spectral (Schatten) decomposition of $Q^{(n)}$ be
\[
 Q^{(n)}=\sum_{i=1}^{\dim \H} q_i^{(n)} E_i^{(n)},\qquad
 (\rank E_i^{(n)}=1)
\]
where the eigenvalues are arranged in the increasing order. 
Take an arbitrary positive number $\lambda$ that is smaller than the minimum positive eigenvalue of $Q^{(\infty)}$.
Then there is an $N\in\N$ and an index $d$, ($1\le d\le \dim\H$), such that for all $n\ge N$,
\[
 q_1^{(n)} \le q_2^{(n)} \le \cdots \le q_{d-1}^{(n)} 
 < \lambda
 < q_{d}^{(n)} \le \cdots \le q_{\dim\H}^{(n)}
\]
and, if $d\ge 2$, then $q_{d-1}^{(n)}\to 0$ as $n\to\infty$. 
Consequently, for $n\ge N$,
\[
 \mathbbm{1}_\lambda(Q^{(n)})=\sum_{i=1}^{d-1} E_i^{(n)} 
 \;\;\mathop{\longrightarrow}_{n\to\infty}\;\;\;
 \sum_{i=1}^{d-1} E_i^{(\infty)}=\mathbbm{1}_\lambda(Q^{(\infty)})
 =\mathbbm{1}_0(Q^{(\infty)}). 
\]

Let us introduce
\[
 O^{(n)}:=\sqrt{\sigma^{(n)}}\, \mathbbm{1}_\lambda(Q^{(n)}) Q^{(n)^{+}} \sqrt{\sigma^{(n)}}.
\]
Then it is shown that $O^{(n)}=o_{L^2}(\rho^{(n)})$. 
In fact, 
\begin{eqnarray*}
\Tr \rho^{(n)} {O}^{(n)^{2}}
&=&\Tr \sigma^{(n)} \mathbbm{1}_\lambda(Q^{(n)}) {Q^{(n)^{+}}} {Q^{(n)^2}} {Q^{(n)^{+}}} \\
&\le& \Tr \sigma^{(n)} \mathbbm{1}_\lambda(Q^{(n)}) \\
&\to& \Tr \sigma^{(\infty)} \mathbbm{1}_0 (Q^{(\infty)}) \\
&=& \Tr \sigma^{(\infty)\perp} \\
&=&0.
\end{eqnarray*}
Here, the inequality follows from
\[
 {Q^{(n)^{+}}} {Q^{(n)^2}} {Q^{(n)^{+}}}
 =\sum_{i: q_i^{(n)} > 0} E_i^{(n)}
 =I-\mathbbm{1}_0(Q^{(n)}),
\]
the second last equality from
\begin{eqnarray*}
 \sigma^{(\infty) ^{ac}}
&=&R^{(\infty)} \rho^{(\infty)} R^{(\infty)} \\
&=& \sqrt{\sigma^{(\infty)}}\, {Q^{(\infty)^{+}}} {Q^{(\infty)^2}} {Q^{(\infty)^{+}}} \sqrt{\sigma^{(\infty)}} \\
&=&\sqrt{\sigma^{(\infty)}} (I-\mathbbm{1}_0(Q^{(\infty)})) \sqrt{\sigma^{(\infty)}},
\end{eqnarray*}
and the last equality from $\sigma^{(\infty)}\ll\rho^{(\infty)}$. 

We next introduce
\[
\overline{R}^{(n)}:={R}^{(n)}-O^{(n)}
=\sqrt{\sigma^{(n)}}\, \left(I-\mathbbm{1}_\lambda(Q^{(n)}) \right) Q^{(n)^{+}}\sqrt{\sigma^{(n)}}.
\]
Then $\overline{R}^{(n)}$ is positive. 
Moreover, it is shown that $\Tr\rho^{(n)}\overline{R}^{(n)^{2}}\to 1$ as $n\to\infty$. 
In fact, 
\begin{equation}\label{eqn:Thm4.4UB}
 \left(I-\mathbbm{1}_\lambda(Q^{(n)}) \right)Q^{(n)^{+}}
 =\left(\sum_{i: q_i^{(n)}>\lambda} E_i^{(n)}\right) 
   \left(\sum_{i: q_i^{(n)}>0} \frac{1}{q_i^{(n)}} E_i^{(n)}\right) 
 = \sum_{i: q_i^{(n)}>\lambda} \frac{1}{q_i^{(n)}} E_i^{(n)},
\end{equation}
which converges to 
\[
 \left(I-\mathbbm{1}_\lambda(Q^{(\infty)}) \right)Q^{(\infty)^{+}}
 =\sum_{i: q_i^{(\infty)}>\lambda} \frac{1}{q_i^{(\infty)}} E_i^{(\infty)}.
\]
In addition, since
\[
 \mathbbm{1}_\lambda(Q^{(\infty)}) Q^{(\infty)^{+}}
 =\left(\sum_{i: q_i^{(\infty)}=0} E_i^{(\infty)}\right) 
   \left(\sum_{i: q_i^{(\infty)}>0} \frac{1}{q_i^{(\infty)}} E_i^{(\infty)}\right) 
 =0,
\]
we have
\begin{equation}\label{eqn:Thm4.4IF}
\left(I-\mathbbm{1}_\lambda(Q^{(n)}) \right)Q^{(n)^{+}} 
\longrightarrow
Q^{(\infty)^{+}}.
\end{equation}
Thus
\[
 \overline{R}^{(n)} \longrightarrow 
 \sqrt{\sigma^{(\infty)}}\,Q^{(\infty)^{+}} \sqrt{\sigma^{(\infty)}}=R^{(\infty)},
\]
so that
\[
\lim_{n\to\infty}\Tr\rho^{(n)}\overline{R}^{(n)^{2}}
=\Tr\rho^{(\infty)}{R}^{(\infty)^{2}}
=\Tr\sigma^{(\infty)}
=1.
\]
Here, the second equality follows from $\sigma^{(\infty)}\ll\rho^{(\infty)}$. 
This identity is combined with $O^{(n)}=o_{L^2}(\rho^{(n)})$ to conclude that $\lim_{n\to\infty}\Tr\rho^{(n)} {R}^{(n)^{2}}=1$. 
Furthermore, due to \eqref{eqn:Thm4.4UB}, the family $\overline{R}^{(n)}$ is uniformly bounded, in that
\[
\overline{R}^{(n)}\leq\frac{1}{\lambda}\sigma^{(n)}\leq\frac{1}{\lambda}.
\]
Thus, the sequence $\overline{R}^{(n)^{2}}$ is uniformly integrable under $\rho^{(n)}$.
This proves $\sigma^{(n)}\vartriangleleft\rho^{(n)}$. 
	
We next prove the `only if' part. 
Let $R^{(n)}$ be a version of the square-root likelihood ratio $ \Ratio\left(\sigma^{(n)}\mid\rho^{(n)}\right)$.
Due to assumption, there is an $L^2$-infinitesimal sequence $O^{(n)}$ of observables such that $\sigma^{(n)}\vartriangleleft_{O^{(n)}}\rho^{(n)}$. 
Let 
\[
 \overline{R}^{(n)}=\sum_{i=1}^{\dim \H} r_i^{(n)} E_i^{(n)},\qquad
 (\rank E_i^{(n)}=1)
\]
be the spectral (Schatten) decomposition of $\overline{R}^{(n)}=R^{(n)}+O^{(n)}$, where the eigenvalues are arranged in the increasing order, so that
\[
 r_1^{(n)} \le r_2^{(n)} \le \cdots \le r_{\dim\H}^{(n)}.
\]
Let us choose the index $d$, ($1\le d\le \dim\H$), that satisfies
\[
 \sup \left\{ \left.r_d^{(n)} \right| n\in\N \right\} <\infty 
 \qquad \mbox{and} \qquad
 \sup \left\{ \left.r_{d+1}^{(n)} \right| n\in\N \right\}=\infty,
\]
and let us define
\[
 A^{(n)}:=\sum_{i=1}^{d} r_i^{(n)} E_i^{(n)} \qquad \mbox{and} \qquad 
 B^{(n)}:=\sum_{i=d+1}^{\dim\H} r_i^{(n)} E_i^{(n)}.
\]
Then $A^{(n)}$ is the uniformly bounded part of $\overline{R}^{(n)}$, and $\overline{R}^{(n)}=A^{(n)}+B^{(n)}$. 

Take a convergent subsequence $A^{(n_{k})}$ of $A^{(n)}$, so that
\[
A_{(\infty)}:=\lim_{k\to\infty}A^{(n_{k})}.
\]
Then for any $M$ that is greater than $M_0:=\sup \left\{ \left.r_d^{(n)} \right| n\in\N \right\}$, 
\[
\lim_{k\to\infty}\overline{R}^{(n_{k})} \mathbbm{1}_{M}(\overline{R}^{(n_{k})})=A_{(\infty)}. 
\]
It then follows from the assumption $\sigma^{(n)}\vartriangleleft_{O^{(n)}}\rho^{(n)}$ that
\begin{equation}\label{eq:trA}
 \Tr\rho^{(\infty)}A_{(\infty)}^{2}
 =\lim_{M\to\infty} \lim_{k\to\infty}\Tr\rho^{(n_{k})}\overline{R}^{(n_{k})^{2}}\mathbbm{1}_{M}(\overline{R}^{(n_{k})})=1.
\end{equation}
Furthermore, since
\[
 \Tr\rho^{(n)}\overline{R}^{(n)^{2}}
 =\Tr\rho^{(n)}(A^{(n)}+B^{(n)})^{2}
 =\Tr\rho^{(n)}A^{(n)^{2}}+\Tr\rho^{(n)}B^{(n)^{2}},
\]
we see that $B^{(n_k)}=o_{L^{2}}(\rho^{(n_{k})})$, 
and so is $C^{(n_{k})}:=R^{(n_k)}-A^{(n_k)}=B^{(n_k)}-O^{(n_k)}$. 
As a consequence, for any unit vector $x\in\H$, 
\begin{eqnarray*}
&&\bracket{x}{R^{(n_k)} \rho^{(n_k)} R^{(n_k)} x} \\
&&\qquad=
\bracket{x}{A^{(n_k)} \rho^{(n_k)} A^{(n_k)}x} 
+2\,{\rm Re}\,\bracket{x}{A^{(n_k)} \rho^{(n_k)} C^{(n_k)}x} 
+\bracket{x}{C^{(n_k)}\rho^{(n_k)} C^{(n_k)}x} \\
&&\qquad \longrightarrow 
\bracket{x}{A_{(\infty)} \rho^{(\infty)} A_{(\infty)}x}
\end{eqnarray*}
as $k\to\infty$. 
In fact
\[
\left| \bracket{x}{C^{(n_k)}\rho^{(n_k)} C^{(n_k)}x}\right| 
\le \Tr C^{(n_k)}\rho^{(n_k)} C^{(n_k)} \longrightarrow 0
\]
and, due to the Schwartz inequality, 
\[
\left|\bracket{x}{A^{(n_k)} \rho^{(n_k)} C^{(n_k)}x} \right|^2
\le 
\bracket{x}{A^{(n_k)} \rho^{(n_k)} A^{(n_k)}x} \bracket{x}{C^{(n_k)} \rho^{(n_k)} C^{(n_k)}x}
\longrightarrow 0.
\]
It then follows from the inequality
\[
\s^{(n_k)} \geq R^{(n_k)}\r^{(n_k)}R^{(n_k)}
\]
that
\[
0\le\bracket{x}{\left(\s^{(n_k)}-R^{(n_k)}\r^{(n_k)}R^{(n_k)}\right)x} 
\mathop{\longrightarrow}_{k\to\infty} \bracket{x}{\left(\s^{(\infty)}-A_{(\infty)} \rho^{(\infty)} A_{(\infty)}\right)x}.
\]
Since $x\in\H$ is arbitrary, we have
\[
 \sigma^{(\infty)} \ge A_{(\infty)}\rho^{(\infty)}A_{(\infty)}.
\]
Combining this inequality with \eqref{eq:trA}, we conclude that
\[
\sigma^{(\infty)}=A_{(\infty)}\rho^{(\infty)}A_{(\infty)}.
\]
This implies that $\sigma^{(\infty)}\ll\rho^{(\infty)}$. 
\end{proof}


\begin{proof}[Proof of Theorem \ref{thm:contiguity_pure}]
We first prove the `if' part. Let
\[
\overline{R}^{(n)}=R^{(n)}
=\sqrt{\sigma^{(n)}}\sqrt{\sqrt{\sigma^{(n)}}\rho^{(n)}\sqrt{\sigma^{(n)}}}^{+}\sqrt{\sigma^{(n)}}.
\]
Due to assumption, there is an $\varepsilon >0$ and $N\in\N$ such that $n\ge N$ implies $\Tr \rho^{(n)}\sigma^{(n)}>\varepsilon$. 
Since $\rho^{(n)}$ is pure, the operator $\sqrt{\sigma^{(n)}}\rho^{(n)}\sqrt{\sigma^{(n)}}$ is rank-one, and its positive eigenvalue is greater than $\varepsilon$. Thus
\[
\overline{R}^{(n)} \leq \frac{1}{\sqrt{\varepsilon}}\sigma^{(n)}\leq\frac{1}{\sqrt{\varepsilon}}
\]
for all $n\ge N$. This implies that $\overline{R}^{(n)}$ is uniformly bounded, so that $\overline{R}^{(n)^{2}}$ is uniformly integrable. 

We next prove the `only if' part. 
Due to assumption, there is an $L^2$-infinitesimal sequence $O^{(n)}$ of observables such that $\sigma^{(n)}\vartriangleleft_{O^{(n)}}\rho^{(n)}$. Let 
\[
 \overline{R}^{(n)}=\sum_i r_i^{(n)} E_i^{(n)}
\]
be the spectral decomposition of $\overline{R}^{(n)}=R^{(n)}+O^{(n)}$, and let $\rho^{(n)}=\ket{\psi^{(n)}}\bra{\psi^{(n)}}$ for some unit vector $\psi^{(n)}\in\H^{(n)}$. 
Since $\lim_{n\to\infty} \Tr \rho^{(n)}R^{(n)^{2}}=1$ is equivalent to $\lim_{n\to\infty} \Tr \rho^{(n)}\overline{R}^{(n)^{2}}=1$, we have
\[
\lim_{n\to\infty}\sum_{i}r_{i}^{(n)^{2}} p_{i}^{(n)}=1, 
\]
where $p_i^{(n)}:=\bracket{\psi^{(n)}}{E_i^{(n)} \psi^{(n)}}$. 
Further, since $\overline{R}^{(n)^2}$ is uniformly integrable, for any $\varepsilon>0$, there exists an $M>0$ such that
\[
\limsup_{n\to\infty}\sum_{i:\,r_{i}^{(n)}>M}r_{i}^{(n)^{2}} p_{i}^{(n)}<\varepsilon.
\]
It then follows that
\begin{eqnarray*}
\liminf_{n\to\infty}\sqrt{\Tr\rho^{(n)}\sigma^{(n)}} 
&\geq& \liminf_{n\to\infty}\sqrt{\Tr\rho^{(n)}R^{(n)}\rho^{(n)}R^{(n)}}\\
&=& \liminf_{n\to\infty}\bra{\psi^{(n)}}R^{(n)}\ket{\psi^{(n)}}\\
&=& \liminf_{n\to\infty}\bra{\psi^{(n)}}\overline{R}^{(n)}\ket{\psi^{(n)}} \\
&=& \liminf_{n\to\infty}\sum_{i}r_{i}^{(n)} p_{i}^{(n)} \\
&\ge& \liminf_{n\to\infty}\sum_{i:\, r_{i}^{(n)}\leq M}r_{i}^{(n)} p_{i}^{(n)}\\
&\ge& \liminf_{n\to\infty}\sum_{i:\,r_{i}^{(n)}\leq M}\frac{r_{i}^{(n)^{2}}}{M} p_{i}^{(n)}\\
& =& \frac{1}{M}\left(1-\limsup_{n\to\infty}\sum_{i:\,r_{i}^{(n)}>M}r_{i}^{(n)^{2}}\ p_{i}^{(n)}\right)\\
& >& \frac{1}{M}\left(1-\varepsilon\right).
\end{eqnarray*}
This completes the proof.
\end{proof}


\begin{proof}[Proof of Lemma \ref{thm:inifini_both_side}]
We shall prove the following series of equalities for any $\{\xi_{t}\}_{t=1}^{r}\subset \R^{d}$ and $\eta_1,\eta_2\in\R$:
\begin{eqnarray*}
&& \lim_{n\to\infty}\Tr\rho^{(n)}e^{\i\eta_{1}\left(Z^{(n)}+O^{(n)}\right)}
\left\{ \prod_{t=1}^{r}e^{\i\xi_{t}^{i}X_{i}^{(n)}}\right\} 
e^{\i\eta_{2}\left(Z^{(n)}+O^{(n)}\right)}\\
&& \qquad=\lim_{n\to\infty}\Tr\rho^{(n)}e^{\i\eta_{1}\left(Z^{(n)}
+O^{(n)}\right)}\left\{ \prod_{t=1}^{r}e^{\i\xi_{t}^{i}X_{i}^{(n)}}\right\} e^{\i\eta_{2}Z^{(n)}}\\
&& \qquad=\lim_{n\to\infty}\Tr\rho^{(n)}e^{\i\eta_{1}Z^{(n)}}\left\{ \prod_{t=1}^{r}e^{\i\xi_{t}^{i}X_{i}^{(n)}}\right\} e^{\i\eta_{2}Z^{(n)}}.
\end{eqnarray*}
The first equality follows from the Schwartz inequality and \eqref{eq:infini_p}: 
\begin{eqnarray*}
&& \left|\Tr\rho^{(n)}e^{\i\eta_{1}\left(Z^{(n)}+O^{(n)}\right)}\left\{ \prod_{t=1}^{r}e^{\i\xi_{t}^{i}X_{i}^{(n)}}\right\} \left\{ e^{\i\eta_{2}\left(Z^{(n)}+O^{(n)}\right)}-e^{\i\eta_{2}Z^{(n)}}\right\} \right|^2 \\
&& \qquad\leq \Tr\rho^{(n)} \left\{ e^{\i\eta_{2}\left(Z^{(n)}+O^{(n)}\right)}-e^{\i\eta_{2}Z^{(n)}}\right\}^* \left\{ e^{\i\eta_{2}\left(Z^{(n)}+O^{(n)}\right)}-e^{\i\eta_{2}Z^{(n)}}\right\} \\
&& \qquad=2-2\,{\rm Re}\,\Tr\rho^{(n)}e^{-\i\eta_{2}\left(Z^{(n)}+O^{(n)}\right)}e^{\i\eta_{2}Z^{(n)}} \\
&& \qquad \longrightarrow 
2-2\,{\rm Re}\,\Tr\rho^{(n)}e^{-\i\eta_{2} Z^{(n)}}e^{\i\eta_{2}Z^{(n)}}=0.
\end{eqnarray*}
The proof of the second equality is similar. 
\end{proof}


\begin{proof}[Proof of Theorem \ref{thm:qLeCam3}]
We first prove that $\psi$ is a well-defined normal state. 
Let $\overline{R}^{(n)}:={R}^{(n)}+{O}^{(n)}$.
It then follows from assumption (ii) and the sandwiched version of the quantum L\'evy-Cram\'er theorem (Lemma \ref{lem:qLevyCramerSand}) that
\begin{eqnarray}
&&\lim_{n\to\infty}\Tr\rho^{(n)}\mathbbm{1}_{M}\left(\overline{R}^{(n)}\right)\overline{R}^{(n)}
\left\{ \prod_{t=1}^{r}e^{\sqrt{-1}\xi_{t}^{i}X_{i}^{(n)}}\right\} \overline{R}^{(n)}
\mathbbm{1}_{M}\left(\overline{R}^{(n)}\right) \label{eq:leCam_third}\\
&&\qquad\qquad =
\phi\left(\mathbbm{1}_{M}\left(R^{(\infty)}\right)R^{(\infty)}\left\{ \prod_{t=1}^{r}e^{\sqrt{-1}\xi_{t}^{i}X_{i}^{(\infty)}}\right\} R^{(\infty)}\mathbbm{1}_{M}\left(R^{(\infty)}\right) \right), \nonumber
\end{eqnarray}
where $M$ is taken to be a non-atomic point of the probability measure $\mu$ having the characteristic function $\varphi_\mu(\eta):=\phi ( e^{\i \eta R^{(\infty)}} )$.
Setting $\xi_t=0$ for all $t$, taking the limit $M\to\infty$, and recalling the uniform integrability of $\overline{R}^{(n)^2}$ as well as the identity $\lim_{n\to\infty} \Tr \r^{(n)} \overline{R}^{(n)^2}=1$, we have 
\begin{equation}\label{eqn:leCam_third_normal}
\lim_{M\to\infty} \phi\left( \mathbbm{1}_{M}(R^{(\infty)}) R^{(\infty)^2} \right)=1.
\end{equation}
Let $\rho$ be the density operator that represents the state $\phi$. 
For notational simplicity, we set $R:=R^{(\infty)}$ and $R_M:=\mathbbm{1}_M(R) R$.
Then, for any $A\in\B(\H^{(\infty)})$,
\[
 \phi(R_MAR_M)
 =\Tr \rho  R_MAR_M
 =\left(R_M\sqrt{\rho}, A R_M \sqrt{\rho} \right)_{\rm HS}, 
\]
where $(B,C)_{\rm HS}:=\Tr B^*C$ is the Hilbert-Schmidt inner product. 
To verify the well-definedness of $\psi$, it suffices to prove that 
$\phi\left(RAR\right)$ exists and 
\[
\phi\left(RAR\right)=
 \lim_{M\to\infty} \phi\left(R_M A R_M\right)
\]
for any $A\in\B(\H^{(\infty)})$. 
To put it differently, it suffices to prove that
$\left\| R\sqrt{\rho} \right\|_{\rm HS}=1$, and that $\left\| R_M\sqrt{\rho}- R\sqrt{\rho} \right\|_{\rm HS} \to 0$ as $M\to\infty$, where $\|\cdot\|_{\rm HS}:=\sqrt{(\,\cdot\,, \,\cdot\,)_{\rm HS}}$. 
Let
\[
 R=\int_0^\infty \lambda \, dE_\lambda
\] 
be the spectral decomposition of $R$, and let $d\nu(\lambda):=\phi(dE_\lambda)$ be the induced probability measure on $\R$. 
It then follows from \eqref{eqn:leCam_third_normal} that
\[
\left\| R\sqrt{\rho} \right\|_{\rm HS}^2
=\Tr\rho R^2
=\int_0^\infty \lambda^2 \,d\nu(\lambda)
=\lim_{M\to\infty} \int_0^M \lambda^2 \,d\nu(\lambda)
=\lim_{M\to\infty} \phi(R_M^2)
=1,
\] 
and that
\[
\left\| R_M\sqrt{\rho}- R\sqrt{\rho} \right\|_{\rm HS}^2
=\Tr \rho R^2 -\Tr \rho R_M^2
=1-\phi(R_M^2)
\longrightarrow 0
\]
as $M\to\infty$.

We next show that for any $\varepsilon>0$ there is an $M>0$ that satisfies
\begin{eqnarray}
&&\sup_{n} \left| \Tr\rho^{(n)} \overline{R}^{(n)}\left\{ \prod_{t=1}^{r}e^{\sqrt{-1}\xi_{t}^{i}X_{i}^{(n)}}\right\} \overline{R}^{(n)} \right.
\label{eqn:thm6.1step1} \\
&&\qquad\qquad\qquad\qquad \left.
-\Tr\rho^{(n)} \mathbbm{1}_{M}\left(\overline{R}^{(n)}\right)\overline{R}^{(n)}\left\{ \prod_{t=1}^{r}e^{\sqrt{-1}\xi_{t}^{i}X_{i}^{(n)}}\right\} \overline{R}^{(n)}\mathbbm{1}_{M}\left(\overline{R}^{(n)}\right) \right|
<\varepsilon. \nonumber
\end{eqnarray}
In fact, 
\begin{eqnarray*}
{\rm (LHS)} & \leq&
		\sup_{n}
		\left|\Tr\rho^{(n)}\overline{R}^{(n)}\left\{ \prod_{t=1}^{r}e^{\sqrt{-1}\xi_{t}^{i}X_{i}^{(n)}}\right\} \left\{ \overline{R}^{(n)}-\overline{R}^{(n)}\mathbbm{1}_{M}\left(\overline{R}^{(n)}\right)\right\} \right|\\
		&& \,\qquad+\sup_{n}
		\left|\Tr\rho^{(n)}\left\{ \overline{R}^{(n)}-\mathbbm{1}_{M}\left(\overline{R}^{(n)}\right)\overline{R}^{(n)}\right\} \left\{ \prod_{t=1}^{r}e^{\sqrt{-1}\xi_{t}^{i}X_{i}^{(n)}}\right\} \overline{R}^{(n)}\mathbbm{1}_{M}\left(\overline{R}^{(n)}\right)\right|,
\end{eqnarray*}
and by using the uniform integrability of $\overline{R}^{(n)^{2}}$,  we see that
\begin{eqnarray*}
\mbox{\rm (first term in RHS)}
\leq \sup_{n}
\sqrt{\Tr\rho^{(n)}\overline{R}^{(n)^{2}}}\sqrt{\Tr\rho^{(n)} \left(I-\mathbbm{1}_{M}(\overline{R}^{(n)}) \right)\overline{R}^{(n)^2}}
<\frac{\varepsilon}{2},
\end{eqnarray*}
and 
\begin{eqnarray*}
\mbox{\rm (second term in RHS)}
\leq \sup_{n}
\sqrt{\Tr\rho^{(n)} \left(I-\mathbbm{1}_{M}(\overline{R}^{(n)}) \right)\overline{R}^{(n)^2}}
\sqrt{\Tr\rho^{(n)} \mathbbm{1}_{M}(\overline{R}^{(n)}) \overline{R}^{(n)^{2}}}
<\frac{\varepsilon}{2}.
\end{eqnarray*}
An important consequence of \eqref{eqn:thm6.1step1} is the following identity
\begin{equation}\label{eq:step_ui2}
\lim_{n\to\infty}\Tr\rho^{(n)}\overline{R}^{(n)}\left\{ \prod_{t=1}^{r}e^{\sqrt{-1}\xi_{t}^{i}X_{i}^{(n)}}\right\} \overline{R}^{(n)}
=\psi \left(\left\{ \prod_{t=1}^{r}e^{\sqrt{-1}\xi_{t}^{i}X_{i}^{(\infty)}}\right\} \right),
\end{equation}
which follows by taking the limit $M\to\infty$ in \eqref{eq:leCam_third}. 

We next observe that
\begin{eqnarray}
\lim_{n\to\infty}\Tr\rho^{(n)}\overline{R}^{(n)}\left\{ \prod_{t=1}^{r}e^{\sqrt{-1}\xi_{t}^{i}X_{i}^{(n)}}\right\} \overline{R}^{(n)} 
&=&
\lim_{n\to\infty}\Tr\rho^{(n)}R^{(n)}\left\{ \prod_{t=1}^{r}e^{\sqrt{-1}\xi_{t}^{i}X_{i}^{(n)}}\right\} \overline{R}^{(n)}
\label{eqn:thm6.1step3} \\
&=&
\lim_{n\to\infty}\Tr\rho^{(n)}R^{(n)}\left\{ \prod_{t=1}^{r}e^{\sqrt{-1}\xi_{t}^{i}X_{i}^{(n)}}\right\} R^{(n)}.
\nonumber
\end{eqnarray}
In fact, the first equality follows from 
\begin{eqnarray*}
\left|\Tr\rho^{(n)}O^{(n)}\left\{ \prod_{t=1}^{r}e^{\sqrt{-1}\xi_{t}^{i}X_{i}^{(n)}}\right\} \overline{R}^{(n)}\right| 
\leq \sqrt{\Tr\rho^{(n)}O^{(n)^{2}}}\sqrt{\Tr\rho^{(n)}\overline{R}^{(n)^{2}}}
\longrightarrow 0,
\end{eqnarray*}
and the second from
\begin{eqnarray*}
\left|\Tr\rho^{(n)}R^{(n)} \left\{ \prod_{t=1}^{r}e^{\sqrt{-1}\xi_{t}^{i}X_{i}^{(n)}} \right\} O^{(n)} \right| 
\leq \sqrt{\Tr\rho^{(n)}R^{(n)^{2}}}\sqrt{\Tr\rho^{(n)}O^{(n)^{2}}}
\longrightarrow 0.
\end{eqnarray*}

We further observe that
\begin{equation}\label{eqn:thm6.1step4}
\lim_{n\to\infty}\Tr\sigma^{(n)}\left\{ \prod_{t=1}^{r}e^{\sqrt{-1}\xi_{t}^{i}X_{i}^{(n)}}\right\} 
=
\lim_{n\to\infty}
\Tr\rho^{(n)}R^{(n)}\left\{ \prod_{t=1}^{r}e^{\sqrt{-1}\xi_{t}^{i}X_{i}^{(n)}}\right\} R^{(n)}. 
\end{equation}
In fact,
\begin{eqnarray*}
\left|\Tr\sigma^{(n)}\left\{ \prod_{t=1}^{r}e^{\sqrt{-1}\xi_{t}^{i}X_{i}^{(n)}}\right\} -\Tr\rho^{(n)}R^{(n)}\left\{ \prod_{t=1}^{r}e^{\sqrt{-1}\xi_{t}^{i}X_{i}^{(n)}}\right\} R^{(n)}\right| 
&\le& 
 \Tr\left|\s^{(n)}-R^{(n)}\rho^{(n)}R^{(n)}\right| \\
&=&
 1-\Tr\rho^{(n)}R^{(n)^{2}}
 \longrightarrow 0.
\end{eqnarray*}

Combining \eqref{eqn:thm6.1step4},  \eqref{eqn:thm6.1step3}, and \eqref{eq:step_ui2}, we have 
\begin{equation}\label{eq:lecam3taeget}
\lim_{n\to\infty}\Tr\sigma^{(n)} \left\{\prod_{t=1}^{r}e^{\sqrt{-1}\xi_{t}^{i}X_{i}^{(n)}} \right\}
=\psi \left(\prod_{t=1}^{r}e^{\sqrt{-1}\xi_{t}^{i}X_{i}^{(\infty)}}\right).
\end{equation}
This completes the proof.
\end{proof}


\begin{proof}[Proof of Theorem \ref{thm:intersect}]
Let
\[
R^{(n)}
=\begin{pmatrix}
0 & 0 & 0\\
0& R_{0}^{(n)} & R_{1}^{(n)} \\
0 & R_{1}^{(n)^{*}}  & R_{2}^{(n)}
\end{pmatrix}
\]
be a version of the square-root likelihood ratio $\ratio{\sigma^{(n)}}{\rho^{(n)}}$ that satisfies
\begin{equation}\label{eq:ex_eq1}
R^{(n)}\rho^{(n)}R^{(n)} 
= 
\begin{pmatrix}
0 & 0 & 0\\
0 & R_{0}^{(n)}\rho_{0}^{(n)}R_{0}^{(n)} & R_{0}^{(n)}\rho_{0}^{(n)}R_{1}^{(n)} \\
0 & R_{1}^{(n)^*}\rho_{0}^{(n)}R_{0}^{(n)} & R_{1}^{(n)^*}\rho_{0}^{(n)}R_{1}^{(n)} 
\end{pmatrix}
\leq\sigma^{(n)}
\end{equation}
and
\begin{equation}\label{eq:ex_eq2}
 \left(\s^{(n)}-R^{(n)}\rho^{(n)}R^{(n)} \right)\perp \r^{(n)}.
 \end{equation}
Since $R_{1}^{(n)^*}\rho_{0}^{(n)}R_{1}^{(n)}\leq\sigma_{2}^{(n)}$ and $\lim_{n\to\infty}\Tr\sigma_{2}^{(n)}=0$,
we see that
\begin{equation}\label{eq:ex_eq3}
 \lim_{n\to\infty}\Tr\rho_{0}^{(n)}R_{1}^{(n)}R_{1}^{(n)^*}=0.
\end{equation}
Further, let 
\[
 \tilde{\sigma}_{0}^{(n)}:=\frac{\sigma_{0}^{(n)}}{\Tr\sigma_{0}^{(n)}},\qquad
 \tilde{\rho}_{0}^{(n)}:=\frac{\rho_{0}^{(n)}}{\Tr\rho_{0}^{(n)}},\qquad
 \tilde{R}_{0}^{(n)}:=\frac{1}{\kappa^{(n)}}R_{0}^{(n)}
\]
where
\[
\kappa^{(n)}=\sqrt{\frac{\Tr\sigma_{0}^{(n)}}{\Tr\rho_{0}^{(n)}}}.
\]
Then it follows from \eqref{eq:ex_eq1} and \eqref{eq:ex_eq2} that $\tilde{R}_{0}^{(n)}\tilde{\rho}_{0}^{(n)}\tilde{R}_{0}^{(n)}\leq\tilde{\sigma}_{0}^{(n)}$
and $\left(\tilde{\sigma}_{0}^{(n)}-\tilde{R}_{0}^{(n)}\tilde{\rho}_{0}^{(n)}\tilde{R}_{0}^{(n)}\right)\perp\tilde{\rho}_{0}^{(n)}$.
This implies that $\tilde{R}_{0}^{(n)}$ is a version of the square-root likelihood ratio $\ratio{\tilde{\sigma}_{0}^{(n)}}{\tilde{\rho}_{0}^{(n)}}$. 
	
The assumption $\tilde{\sigma}_{0}^{(n)}\vartriangleleft\tilde{\rho}_{0}^{(n)}$ ensures the existence of 
a sequence $O_{0}^{(n)}=o_{L^2}(\tilde{\rho}_{0}^{(n)})$ 
such that $\tilde{\sigma}_{0}^{(n)}\vartriangleleft_{O_{0}^{(n)}}\tilde{\rho}_{0}^{(n)}$.
Let $\overline{R}_{0}^{(n)}:=\tilde{R}_{0}^{(n)}+O_{0}^{(n)}$, and let
\[
\overline{R}^{(n)}=\begin{pmatrix}
0 & 0 & 0\\
0 & \kappa^{(n)} \overline{R}_{0}^{(n)} & 0\\
0 & 0 & 0
\end{pmatrix}.
\]
Then we see that
\begin{eqnarray*}
O^{(n)} 
:= \overline{R}^{(n)}-R^{(n)}
=\begin{pmatrix}
0 & 0 & 0\\
0& \kappa^{(n)}O_{0}^{(n)} & -R_{1}^{(n)} \\
0 & -R_{1}^{(n)^{*}}  & -R_{2}^{(n)}
\end{pmatrix}
\end{eqnarray*}
is $L^{2}$-infinitesimal with respect to $\rho^{(n)}$.
In fact, due to \eqref{eq:ex_eq3}, 
\[
\lim_{n\to\infty} \Tr\rho^{(n)}O^{(n)^{2}}
=
\lim_{n\to\infty} \Tr\rho_{0}^{(n)}\left\{\kappa^{(n)^{2}}O_{0}^{(n)^{2}}+R_{1}^{(n)}R_{1}^{(n)^*} \right\}
=0. 
\]
Furthermore, 
\[
\lim_{n\to\infty}\Tr\rho^{(n)}\overline{R}^{(n)^{2}} 
=
\lim_{n\to\infty}\kappa^{(n)^2}\Tr\rho_{0}^{(n)}\overline{R}_{0}^{(n)^{2}}
=
\lim_{n\to\infty} (\Tr\sigma_{0}^{(n)})\Tr\tilde{\rho}_{0}^{(n)}\overline{R}_{0}^{(n)^{2}}
=1,
\]
and
\begin{eqnarray*}
\lim_{M\to\infty}\liminf_{n\to\infty}\Tr\rho^{(n)}\overline{R}^{(n)^{2}}\mathbbm{1}_{M}(\overline{R}^{(n)})
&=& 
\lim_{M\to\infty}\liminf_{n\to\infty}\kappa^{(n)^{2}}\Tr\rho_{0}^{(n)}\overline{R}_{0}^{(n)^{2}}\mathbbm{1}_{M}( \kappa^{(n)} \overline{R}_0^{(n)}) \\
&=& 
\lim_{M\to\infty}\liminf_{n\to\infty}(\Tr\sigma_{0}^{(n)})\Tr\tilde{\rho}_{0}^{(n)}\overline{R}_{0}^{(n)^{2}} \mathbbm{1}_{M/\kappa^{(n)}}(\overline{R}_0^{(n)}) \\
&\geq&
 \lim_{M\to\infty}\liminf_{n\to\infty}(\Tr\sigma_{0}^{(n)})\Tr\tilde{\rho}_{0}^{(n)}\overline{R}_{0}^{(n)^{2}} \mathbbm{1}_{\lambda M}(\overline{R}_0^{(n)})
=1,
\end{eqnarray*}
where 
\[
\lambda:=\liminf_{n\to\infty}\frac{1}{\kappa^{(n)}} = \liminf_{n\to\infty} \sqrt{\Tr \r^{(n)}_0} > 0.
\] 
Thus $\sigma^{(n)}\vartriangleleft_{O^{(n)}}\rho^{(n)}$. 
\end{proof}


\begin{proof}[Proof of Theorem \ref{thm:kakutani}]
We first prove the `only if' part. 
Due to assumption, there is an $L^2$-infinitesimal sequence $O^{(n)}$ of observables satisfying the condition that for any $\varepsilon>0$, there is an $M>0$ such that
\[
 \liminf_{n\to\infty}\Tr\rho^{(n)} \mathbbm{1}_{M}(\overline{R}^{(n)})\overline{R}^{(n)^{2}}>1-\varepsilon,
\]
where $\overline{R}^{(n)}:=R^{(n)}+O^{(n)}$ with $R^{(n)}:=\bigotimes_{i=1}^n R_i$. 
It then follows that
\begin{eqnarray*}
\prod_{i=1}^\infty \Tr\rho_i R_i
&=& \lim_{n\to\infty} \Tr \rho^{(n)} R^{(n)} \\
&=& \lim_{n\to\infty} \Tr \rho^{(n)} \overline{R}^{(n)} \\
&\ge&\liminf_{n\to\infty}  \Tr \rho^{(n)} \overline{R}^{(n)} \mathbbm{1}_{M}(\overline{R}^{(n)})\\
&\ge&\liminf_{n\to\infty}  \Tr \rho^{(n)} \frac{\overline{R}^{(n)^2}}{M} \mathbbm{1}_{M}(\overline{R}^{(n)})\\
&>& \frac{1}{M}(1-\varepsilon).
\end{eqnarray*}
Further, the equivalence of \eqref{eqn:kakutaniCriterion} and \eqref{eqn:kakutaniCriterion2} is well known, (see \cite[Section 14.12]{Williams}, for example).

We next prove the `if' part. 
Since $\sigma^{(n)}\ll \rho^{(n)}$,  we have $\Tr \rho^{(n)} R^{(n)^2}=1$ for all $n$. 
It then suffices to prove that ${R}^{(n)^{2}}$ is uniformly integrable under $\rho^{(n)}$. 
For each $i\in\N$, let
\[
 R_i=\sum_{x\in\X_i} r_i(x) \ket{\psi_i(x)}\bra{\psi_i(x)}
\]
be a Schatten decomposition of $R_i$, where $\X_i=\{1,\dots,\dim\H_i\}$ is a standard reference set that put labels on the eigenvalues $r_i(x)$ and eigenvectors $\psi_i(x)$.
Note that the totality $\{\psi_i(x)\}_{x\in\X_i}$ of eigenvectors forms an orthonormal basis of $\H_i$. 
Let
\[
 p_i(x):=\bracket{\psi_i(x)}{\rho_i \psi_i(x)},\qquad
 q_i(x):=\bracket{\psi_i(x)}{\sigma_i \psi_i(x)}.
\]
Then $P_i:=(p_i(x))_{x\in\X_i}$ and $Q_i:=(q_i(x))_{x\in\X_i}$ are regarded as classical probability distributions on $\X_i$. 
Due to the identity $\sigma_i=R_i\rho_i R_i$, we have
\[
  q_i(x)=p_i(x) r_i(x)^2,\qquad (\forall x\in\X_i),
\]
which implies that $Q_i \ll P_i$ for all $i\in\N$.
Now, since
\[
 \Tr\rho_iR_i=\sum_{x\in\X_i} p_i(x) r_i(x)=\sum_{x\in\X_i} \sqrt{p_i(x)q_i(x)},
\]
assumption \eqref{eqn:kakutaniCriterion} is equivalent to 
\[
\prod_{i=1}^\infty \left(\sum_{x\in\X_i} \sqrt{p_i(x)q_i(x)}\right)>0.
\]
This is nothing but the celebrated Kakutani criterion for the infinite product measure $\prod_i Q_i$ to be absolutely continuous to $\prod_i P_i$, (cf. \cite{{Kakutani},{Williams}}). 
As a consequence, the classical likelihood ratio process
\[
 L^{(n)}(X_1,\dots,X_n):=\prod_{i=1}^n \,\frac{q_i(X_i)}{p_i(X_i)}
\]
is uniformly integrable under $\prod_i P_i$, (cf. \cite[Section 14.17]{Williams}).
The uniform integrability of $R^{(n)^2}$ under $\rho^{(n)}$ now follows immediately from the identity
\[
 \Tr \rho^{(n)} \mathbbm{1}_{M}({R}^{(n)}){R}^{(n)^{2}}
 =E_{P^{(n)}}\left[ \mathbbm{1}_{M^2} ({L}^{(n)}){L}^{(n)} \right],
\]
where $P^{(n)}:=\prod_{i=1}^n P_i$. 
\end{proof}


\begin{proof}[Proof of Theorem \ref{thm:iid}]
Since the symmetric logarithmic derivative $L_i$ at $\theta_0$ satisfies $\Tr\rho_{\theta_0} L_i=0$ for all $i\in\{1,\dots, d\}$, the property (i) in Definition \ref{def:QLAN} is an immediate consequence of an i.i.d. version of the quantum central limit theorem \cite{{qclt},{YFG}}.

In order to prove (ii) in Definition \ref{def:QLAN}, we first calculate the square-root likelihood ratio $\ratio{\r_{\th}^{\otimes n}}{\r_{\th_0}^{\otimes n}}$ 
between $\rho_{\theta}^{\otimes n}$ and $\rho_{\theta_0}^{\otimes n}$. 
Let $\r_\th=\r_\th^{ac}+\r_\th^\perp$ be the Lebesgue decomposition with respect to $\r_{\th_0}$. 
Then
\begin{equation}\label{eqn:absoluteContinuity-n}
\rho_{\theta}^{\otimes n} 
\ge 
\left(\rho_{\theta}^{ac}\right)^{\otimes n} 
=\left(R_{\th} \rho_{\theta_{0}} R_{\th} \right)^{\otimes n}
=R_{\th}^{\otimes n} \,\rho_{\theta_{0}}^{\otimes n}\, R_{\th}^{\otimes n},
\end{equation}
where $R_{\th}=\ratio{\r_{\th}}{\r_{\th_0}}$. 
On the other hand,
\[
\Tr \r_{\th_0}\r_\th= \Tr \r_{\th_0}\r_\th^{ac}+\Tr \r_{\th_0}\r_\th^\perp=\Tr \r_{\th_0}\r_\th^{ac}
=\Tr \rho_{\theta_0} \left(R_{\th}\rho_{\theta_{0}}R_{\th}\right).
\]
Therefore, 
\begin{eqnarray*}
\Tr \rho_{\theta_0}^{\otimes n} 
\left[ \rho_{\theta}^{\otimes n} - \left(R_{\th}\rho_{\theta_{0}}R_{\th}\right)^{\otimes n} \right] 
=
(\Tr \rho_{\theta_0} \rho_{\theta})^n - \left(\Tr \rho_{\theta_0} \left(R_{\th}\rho_{\theta_{0}}R_{\th}\right) \right)^n=0.
\end{eqnarray*}
Due to Lemma \ref{lem:1}, this implies that
\begin{equation}\label{eqn:singular-n} 
\rho_{\theta_0}^{\otimes n} \perp \left[ \rho_{\theta}^{\otimes n} - \left(R_{\th}\rho_{\theta_{0}}R_{\th} \right)^{\otimes n} \right].
\end{equation}
From \eqref{eqn:absoluteContinuity-n} and \eqref{eqn:singular-n}, we have the quantum Lebesgue decomposition
\[
\r_\th^{\otimes n}=(\r_\th^{\otimes n})^{ac}+(\r_\th^{\otimes n})^\perp
\] 
with respect to $\r_{\th_0}^{\otimes n}$, where
\[
(\r_\th^{\otimes n})^{ac}
=R_{\th}^{\otimes n}\,\rho_{\theta_{0}}^{\otimes n}\, R_{\th}^{\otimes n}
\qquad\mbox{and}\qquad
(\r_\th^{\otimes n})^\perp
=\rho_{\theta}^{\otimes n} - R_{\th}^{\otimes n}\,\rho_{\theta_{0}}^{\otimes n}\, R_{\th}^{\otimes n}.
\]
Consequently, $R_{\th}^{\otimes n}$ gives a version of the square-root likelihood ratio $\ratio{\r_{\th}^{\otimes n}}{\r_{\th_0}^{\otimes n}}$.

Let us proceed to the proof of (ii) in Definition \ref{def:QLAN}. 
Since $R_{h}$ is differentiable at $h=0$ and $R_{0}=I$, it is expanded as 
\[
R_{h}=I+\frac{1}{2}A_{i}h^{i}+o(\left\Vert h\right\Vert ).
\]
Due to assumption \eqref{eq:oh2},
\begin{eqnarray*}
\rho_{\theta_{0}+h} 
= R_{h}\rho_{\theta_{0}}R_{h}+o(\left\Vert h\right\Vert ^{2})
= \rho_{\theta_{0}}+\frac{1}{2}\left(A_{i}\rho_{\theta_{0}}+\rho_{\theta_{0}}A_{i}\right)h^{i}+o(\left\Vert h\right\Vert ).
\end{eqnarray*}
As a consequence,
the selfadjoint operator $A_i$ is also a version of the $i$th SLD at $\theta_0$.
To evaluate the higher order term of $R_h$, let
\[
B(h):=R_{h}-I-\frac{1}{2}A_{i}h^{i}.
\]
Then
\begin{eqnarray*}
\Tr\rho_{\theta_{0}}R_{h}^{2} 
&=& \Tr\rho_{\theta_{0}}\left(I+\frac{1}{2}A_{i}h^{i}+B(h)\right)^{2}\\
&=& \Tr\rho_{\theta_{0}}\left(I+\frac{1}{4}A_{i}A_{j}h^{i}h^{j}+2B(h)+A_{i}h^{i}+B(h)^{2}+\frac{1}{2}A_{i}h^{i}B(h)+\frac{1}{2}B(h)A_{i}h^{i}\right)\\
&=& 1+\frac{1}{4}J_{ji}h^{i}h^{j}+2\Tr\rho_{\theta_{0}}B(h)+o(\left\Vert h\right\Vert ^{2}).
\end{eqnarray*}
This relation and assumption \eqref{eq:oh2} lead to 
\begin{equation}
\Tr\rho_{\theta_{0}}B(h)=-\frac{1}{8}J_{ji}h^{i}h^{j}+o(\left\Vert h\right\Vert ^{2})\label{eq:Bh}.
\end{equation}
In order to prove (ii), it suffices to show that
\begin{eqnarray*}
O_{h}^{(n)}
&:=&\exp\left[\frac{1}{2}\left(h^{i}\Delta_{i}^{(n)}-\frac{1}{2}J_{ji}h^{i}h^{j}\right) \right]
-(R_{h/\sqrt{n}})^{\otimes n}\\
& =&e^{-\frac{1}{4}J_{ji}h^{i}h^{j}}\left\{ e^{\frac{1}{2\sqrt{n}}h^{i}L_{i}}\right\} ^{\otimes n}-(R_{h/\sqrt{n}})^{\otimes n}
\end{eqnarray*}
is $L^{2}$-infinitesimal under $\rho_{\theta_{0}}^{\otimes n}$, setting the D-infinitesimal residual term $o_D \left(h^i \Delta^{(n)}_i,\rho_{\theta_{0}}^{(n)}\right)$ in (ii) to be zero for all $n$.
In fact, 
\begin{eqnarray}
\Tr\rho_{\theta_{0}}^{\otimes n}O_{h}^{(n)^{2}} 
& =&e^{-\frac{1}{2}J_{ji}h^{i}h^{j}}\left\{ \Tr\rho_{\theta_{0}}e^{\frac{1}{\sqrt{n}}h^{i}L_{i}}\right\} ^{n}
+\left\{ \Tr\rho_{\theta_{0}}R_{h/\sqrt{n}}^{2}\right\} ^{n}
\label{eq:four_term} \\
&& 
-2\,e^{-\frac{1}{4}J_{ji}h^{i}h^{j}}\,
{\rm Re}\left\{ \Tr\rho_{\theta_{0}}e^{\frac{1}{2\sqrt{n}}h^{i}L_{i}}R_{h/\sqrt{n}}\right\} ^{n}.
\nonumber
\end{eqnarray}
The first term in the right-hand side of \eqref{eq:four_term} is evaluated as follows:
\begin{eqnarray*}
e^{-\frac{1}{2}J_{ji}h^{i}h^{j}}\left\{ \Tr\rho_{\theta_{0}}e^{\frac{1}{\sqrt{n}}h^{i}L_i}\right\} ^{n} 
&=& e^{-\frac{1}{2}J_{ji}h^{i}h^{j}}\left\{ \Tr\rho_{\theta_{0}}\left(I+\frac{1}{\sqrt{n}}h^{i}L_{i}
+\frac{1}{2n}L_{i}L_{j}h^{i}h^{j}+o\left(\frac{1}{n}\right)\right)\right\} ^{n}\\
& =& e^{-\frac{1}{2}J_{ji}h^{i}h^{j}}
\left(1+\frac{1}{2n}J_{ji}h^{i}h^{j}+o\left(\frac{1}{n}\right)\right)^{n}\longrightarrow 1. 
\end{eqnarray*}
The second term is evaluated from \eqref{eq:oh2} as
\[
\left\{ \Tr\rho_{\theta_{0}}R_{h/\sqrt{n}}^{2}\right\} ^{n}
=\left(1-o\left(\frac{1}{n}\right)\right)^{n}\longrightarrow 1.
\]
Finally, the third term is evaluated from \eqref{eq:Bh} as
\begin{eqnarray*}
&&e^{-\frac{1}{4}J_{ji}h^{i}h^{j}}\left\{ \Tr\rho_{\theta_{0}}e^{\frac{h^{i}}{2\sqrt{n}}L_{i}}R_{h/\sqrt{n}}\right\} ^{n} \\
&&\qquad =e^{-\frac{1}{4}J_{ji}h^{i}h^{j}}\left\{ \Tr\rho_{\theta_{0}}\left(I+\frac{h^{i}}{2\sqrt{n}}L_{i}+\frac{1}{8n}L_{i}L_{j}h^{i}h^{j}+o\left(\frac{1}{n}\right)\right)
\left(I+\frac{h^{k}}{2\sqrt{n}} A_{k}+B\left(\frac{h}{\sqrt{n}}\right)\right)\right\} ^{n}\\
&& \qquad =e^{-\frac{1}{4}J_{ji}h^{i}h^{j}}
\left\{ 1+\frac{1}{4n}J_{ki}h^{i}h^{k}+o\left(\frac{1}{n}\right)\right\} ^{n}
\longrightarrow 1.
\end{eqnarray*}
This proves (ii). 

Having established that $\{\rho_\theta^{\otimes n}\}_n$ is q-LAN at $\theta_0$, 
the property \eqref{eq:iidLeCam3} is now an immediate consequence of Corollary \ref{cor:LeCam3qlan} as well as the quantum central limit theorem
\begin{equation}\label{eq:qcovergenceTogetherIID}
\begin{pmatrix}X^{(n)}\\
\Delta^{(n)}
\end{pmatrix}
\convd{\;\; \rho_{\theta_{0}}^{\otimes n}}
N\left(\begin{pmatrix}0\\
0
\end{pmatrix},\begin{pmatrix}\Sigma & \tau\\
\tau* & J
\end{pmatrix}\right).
\end{equation}
This completes the proof.
\end{proof}

\end{document}